\documentclass{article}

\def\cal{\mathcal}
 \catcode`@=11
\usepackage{color}
\usepackage{amsthm,amssymb,amsfonts,mathrsfs}
\usepackage{amsmath}
\catcode`@=11 \@addtoreset{equation}{section}

\catcode`@=12
\newtheorem{Theorem}{Theorem}[section]
\newtheorem{Proposition}{Proposition}[section]
\newtheorem{Lemma}{Lemma}[section]
\newtheorem{Corollary}{Corollary}[section]
\theoremstyle{definition}

\newtheorem{Definition}{Definition}[section]
\newtheorem{Remark}{Remark}

\newtheorem{Assumptions}{Hypothesis}[section]
\def\R{{\mathbb{R}}}

\newcommand{\ds}{\displaystyle}

\def\fg{\mathfrak{g}}
\def\fh{\mathfrak{h}}

\def\fd{\mathfrak{d}}
\def\fg{\mathfrak{g}}
\def\fh{\mathfrak{h}}

\title{Carleman estimates for parabolic equations with interior degeneracy and Neumann boundary conditions}
\author{
{\sc Idriss Boutaayamou}\\
D\'{e}partement de Math\'{e}matiques\\ Facult\'{e} des Sciences Semlalia\\
LMDP, UMMISCO (IRD-UPMC)\\
Universit\'{e} Cadi Ayyad,\\ Marrakech 40000, B.P. 2390, Morocco\\ email: dsboutaayamou@gmail.com
\\
{\sc Genni Fragnelli\thanks{The author is a member of the Gruppo Nazionale per l'Analisi Matematica, la Probabilit\`a e le loro Applicazioni (GNAMPA) of the
Istituto Nazionale di Alta Matematica (INdAM) and she is partially supported by the reaserch project {\em Comportamento asintotico e controllo di equazioni di evoluzione non lineari}
of the GNAMPA-INdAM.
}}\\
Dipartimento di Matematica\\ Universit\`{a} di Bari "Aldo Moro"\\
Via
E. Orabona 4\\ 70125 Bari - Italy\\ email: genni.fragnelli@uniba.it\\
{\sc  Lahcen Maniar}\\
D\'{e}partement de Math\'{e}matiques\\ Facult\'{e} des Sciences Semlalia\\
LMDP, UMMISCO (IRD-UPMC)\\
Universit\'{e} Cadi Ayyad,\\ Marrakech 40000, B.P. 2390, Morocco\\ email: maniar@uca.ma}
\date{}
\begin{document}

\maketitle

\begin{abstract}
We consider a parabolic problem with degeneracy in the interior of
the spatial domain and Neumann boundary conditions. In particular, we will focus on the well-posedness of the problem and on
Carleman estimates for the associated adjoint problem. The novelty
of the present paper is that for the first time it is considered a problem with an interior degeneracy and Neumann boundary conditions so that no previous
result can be adapted to this situation. As a consequence new observability inequalities are
established.
\end{abstract}



\section{Introduction}
The study of degenerate parabolic equations is the subject of numerous
articles and books. Indeed many problems coming from
physics, biology and economics
are described by
degenerate parabolic equations, whose linear prototype is
\begin{equation}\label{0}
\frac{\partial u}{\partial t} - Au  =h(t,x), \quad (t,x) \in (0,T) \times (0,1)
\end{equation}
with the associated desired boundary conditions. Here $T>0$ is given, $h$ belongs to a suitable Lebesgue space and ${A}u= { A}_1u:= (au_x)_x$ or ${A}u = { A}_2u:= au_{xx}$, where $a$ is a degenerate function.

In the present paper we will focus on a particular topic related to this
field of research, i.e. Carleman estimates for the adjoint problem
of the previous equation. Indeed, they have so many applications that a large
number of papers has been devoted to prove some forms of them and
possibly some applications. For example, it is well known that they
are crucial
 for inverse
problems (see, for example, \cite{salo}) and
for unique continuation
properties (see, for example,  \cite{LRL}).
In particular, they are a fundamental tool to prove observability inequalities, which
lead to global null controllability results for  the Cauchy problem associated to \eqref{0} also in
the non degenerate case (see, for instance, \cite{m0} - \cite{acf}, \cite{cf} - \cite{cfv1},
\cite{fcgb} - \cite{fm1}, \cite{LRL}, \cite{vz} and the references therein). For related systems of degenerate equations we refer, for example, to \cite{m0} and
\cite{m}.

 In most of the previous papers the authors assume that the function $a$ degenerates at the boundary of the space domain, for example $a(x)= x^k(1-x)^\alpha,$ $x \,\in \,[0,1],$ where $k$ and $\alpha$ are positive constants, and the degeneracy is {\it regular} .The question of Carleman estimates for partial differential systems with non smooth
coefficients, i.e. the coefficient $a$ is not of class $C^1$ (or even with higher regularity,
as sometimes it is required) is
not fully solved yet. Indeed, the presence of a non smooth coefficient introduces
several complications, and, in fact, the literature in this context is quite poor also in the non degenerate case (for more details see \cite{fm1}). To our best knowledge, the first results on Carleman estimates for the adjoint problem of \eqref{0} with an interior degenerate point are obtained in \cite{fm}, for a regular degeneracy, and in \cite{fm1}, for a
globally
non smooth degeneracy. We underline that in \cite{fm} and in \cite{fm1} the authors consider the problem in divergence (\cite{fm}, \cite{fm1}) or in non divergence form (\cite{fm1}) but only with {\it Dirichlet boundary conditions}. We also refer to \cite{bfm}, where an inverse source problem of a $2 \times 2$ cascade parabolic systems with interior
degeneracy is studied.

However, in all the previous papers the authors consider \eqref{0} only with {\it Dirichlet boundary conditions}. Neumann boundary conditions are considered in \cite{acf} and in \cite{f}, but again the degeneracy is at the boundary of the space domain.

The goal of this paper is to give a full analysis of \eqref{0} with {\it Neumann boundary conditions }in the case that the degeneracy occurs at the interior of the space domain; moreover, the coefficient is allowed to be {\it non smooth} in the non divergence case and in the strongly degenerate divergence case. In particular,
we consider the following problem:
\begin{equation}\label{linear}
\begin{cases}
\displaystyle{\frac{\partial u}{\partial t}} - Au  =h(t,x), & (t,x) \in Q_T,\\
u_x(t,0)=u_x(t,1)=0, & t \ge 0, \\
u(0,x)=u_0(x), & x \in (0,1),
\end{cases}
\end{equation}
where
$Q_T:=(0,T) \times (0,1)$, $Au:=A_1u:=(au_x)_x$ or $Au:=A_2u:=au_{xx}$,
$a$ degenerates at $x_0\in (0,1)$, $u_0 \in X$ and $h \in L^2(0,T;X)$. Here $X$ denotes the Hilbert space $L^2(0,1)$, in the divergence form, and  $L^2_{\frac{1}{a}}(0,1)$, in the non divergence one (for the precise definition of $L^2_{\frac{1}{a}}(0,1)$ we refer to Section \ref{non_divergence}).

We give the following definitions:
\begin{Definition}\label{def1}  The operators $A_1u
:=(au')'$ and $A_2u= au''$ are weakly degenerate if there exists $x_0 \in (0,1)$ such
that $a(x_0)=0$, $a>0$ on $[0, 1]\setminus \{x_0\}$, $a\in
W^{1,1}(0,1)$ and there exists $K_1 \in (0,1)$ such that $(x-x_0)a'
\le K_1 a$ a.e. in $[0,1]$.
\end{Definition}

\begin{Definition}\label{def2}
The operators $A_1u :=(au')'$ and $A_2u= au''$ are strongly
degenerate if there exists $x_0 \in (0,1)$
such that $a(x_0)=0$, $a>0$ on $[0, 1]\setminus \{x_0\}$, $a\in
W^{1, \infty}(0,1)$ and there exists $K_2 \in [1,2)$ such that
$(x-x_0)a' \le K_2 a$ a.e. in $[0,1]$.
\end{Definition}

Typical examples for weak and strong degeneracies are $a(x)=|x-
x_0|^{\alpha}, \; 0<\alpha<1$ and $a(x)= |x- x_0|^{\alpha}, \;
1\le\alpha<2$, respectively.

\medskip

The object of this paper is twofold:  first we analyze the well-posedness of the problem with {\it Neumann boundary conditions}; second we prove Carleman estimates. To this aim we have a new approach: first, we use a reflection procedure and then we employ the Carleman estimates for the analogue of \eqref{linear} with Dirichlet boundary conditions proved in \cite{fm1}.
Finally, as a consequence of the Carleman estimates we prove, using again a reflection procedure, observability inequalities. In particular, we prove that there exists a positive constant $C_T$ such that every solution
$v$ of
\[
\begin{cases}
v_t +Av= 0, &(t,x) \in  Q_T,
\\[5pt]
v_x(t,0)=v_x(t,1) =0, & t \in (0,T),
\\[5pt]
v(T,x)= v_T(x)\in X,
\end{cases}
\]
satisfies, under suitable assumptions, the following estimate:
 \begin{equation}\label{obs}
\|v(0)\|^2_X \le C_T\|v\chi_\omega\|^2_{L^2(0,T;X)}.
\end{equation}
Here $\chi_\omega$ is the characteristic function of the control region $\omega$ which is assumed to be an interval which contains the degeneracy point
or
an interval lying on one
side of the degeneracy point. As an immediate consequence, we can prove, using a standard
technique (e.g., see \cite[Section 7.4]{LRL}),  the
null controllability result for the linear degenerate problem:
if \eqref{obs} holds, then for
every $u_0 \in X$ there exists $h \in L^2(0,T;X)$ such that
the solution $u$ of
\begin{equation}\label{linear1}
\begin{cases}
\displaystyle{\frac{\partial u}{\partial t}} - Au  =h(t,x)\chi_\omega(x), & (t,x) \in Q_T,\\
u_x(t,0)=u_x(t,1)=0, & t \ge 0, \\
u(0,x)=u_0(x), & x \in (0,1),
\end{cases}
\end{equation}
is such that
$
u(T,x)= 0 \ \text{ for every  } \  x \in [0, 1];
$
moreover
$
\|h\|^2_{L^2(0,T;X)} \le C \|u_0\|^2_X,
$
for some universal positive constant $C$.

We conclude this introduction underlining the fact that in the present paper we consider
equations {\em in divergence} and {\em in non divergence form}, since the last one {\it cannot} be recast in
divergence form: for example, the simple equation
$u_t=a(x)u_{xx}$
can be written in divergence form as
$
u_t=(a(x)u_x)_x-a'u_x,
$
only if $a'$ does exist; in addition, as far as well-posedness is
considered for the last equation, additional conditions are
necessary. For instance, for the prototype $a(x)= |x-x_0|^K$
well-posedness is guaranteed if $K \ge 2$. However, in
\cite{fm1} the authors prove that if $a(x)=|x-x_0|^K$ the global null
controllability fails exactly when $K\ge 2$.

\medskip

 The paper is organized as follows: in Sections $2$ and $3$ we study the well-posedness of the problem and we characterize the domain of the operator in some cases. In Sections $4$ and $5$ we prove Carleman estimates for the problem in divergence and in non divergence form. As a consequence, in Section $6$, we prove observability inequalities and we conclude the paper with some comments on Carleman estimates.


A final comment on the notation: by $C$ and $C_T$ we shall denote
universal positive constants, which are allowed to vary from line to
line and depend only on the coefficients of the equation.

\section{Well posedness in the divergence case}\label{divergence_form}
In this section we consider the operator in divergence form, that is
$A_1u=(au')'$, and we distinguish, as usual, two cases: the weakly degenerate
case and the strongly degenerate one.
\subsection{Weakly degenerate operator}
Throughout this subsection we assume that the operator is weakly
degenerate.
\\
In order to prove that $A_1$, with a suitable domain, generates a
strongly continuous semigroup, we introduce, as in \cite{acf} or \cite{fggr}, the
following weighted spaces:
\[
 H^1_a (0,1):=\{ u  \text{ is absolutely
continuous in } [0,1] \text{ and }\sqrt{a} u' \in  L^2(0,1) \}
\]
with the norm
\begin{equation}\label{norm}
\|u\|_{H^1_a(0,1)}^2:= \|u\|^2_{L^2(0,1)} +
\|\sqrt{a}u'\|^2_{L^2(0,1)}
\end{equation}
and $
H^2_a(0,1) := \{ u \in H^1_a(0,1) |\,au' \in H^1(0,1)\}$
with
\begin{equation}\label{norm1}
\|u\|_{H^2_a(0,1)}^2:= \|u\|^2_{H^1_a(0,1)} +
\|(au')'\|^2_{L^2(0,1)}.
\end{equation}
 Then,
define the operator $A_1$ by $D(A_1)= \{ u \in H^2_a(0,1)| \, u'(0)=u'(1)=0\},$
 and, for any $ \,u \in D(A_1),$
$ A_1u:=(au')'.
$
As in \cite[Lemma 2.1]{fggr}, using the fact that $u'(0)=u'(1)=0$ for all $u \in D(A_1)$, one can prove the following formula of integration by parts:
\begin{Lemma}\label{green} For all $(u,v)\in D(A_1)\times H^1_a(0,1)$ one has
\begin{equation}\label{greenformula}
\int_0^1(au')' v dx= - \int_0^1 au'v' dx.
\end{equation}
\end{Lemma}

Now, let us go back to problem
\eqref{linear}, recalling the following
\begin{Definition}
If $u_0 \in L^2(0,1)$ and $h\in L^2(Q_T):= L^2(0,T; L^2(0,1))$, a function $u$ is said to
be a weak solution of \eqref{linear} with $A= A_1$ if
\[
u \in C([0, T]; L^2(0,1)) \cap L^2(0, T; { H}^1_a(0,1))
\]
and
\[
\begin{aligned}
&\int_0^1u(T,x)\varphi(T,x)\, dx - \int_0^1 u_0(x) \varphi(0,x)\, dx
- \int_{Q_T}u\varphi_t \,dxdt =
\\&- \int_{Q_T} au_x
\varphi_x\,dxdt + \int_{Q_T} h\varphi \,dx dt
\end{aligned}
\]
for all $\varphi \in H^1(0, T; L^2(0,1)) \cap L^2(0, T; {
H}^1_a(0,1))$.
\end{Definition}
Hence, the next result holds.
\begin{Theorem}\label{prop} The operator $A_1: D(A_1) \to L^2(0,1)$ is self--adjoint,
nonpositive on $L^2(0,1)$ and it generates an analytic
contraction semigroup of angle $\pi/2$. Therefore, for all $h
\in L^2(Q_T)$ and $u_0 \in
L^2(0,1)$, there exists a unique solution
\[u \in
C\big([0,T]; L^2(0,1)\big) \cap L^2 \big(0,T;
H^1_a(0,1)\big)\]
of \eqref{linear} such that
\begin{equation}\label{stima2w}
\sup_{t \in [0,T]}
\|u(t)\|^2_{L^2(0,1)}+\int_0^T\|u(t)\|^2_{H^1_a
(0,1)} dt \le
C_T\left(\|u_0\|^2_{L^2(0,1)}+\|h\|^2_{L^2(Q_T)}\right),
\end{equation}
for some positive constant $C_T$. Moreover, if $h \in W^{1,1}(0,T; L^2(0,1))$ and $u_0 \in
H^1_a(0,1)$, then
\begin{equation}\label{regularity1}
u\in C^1\big([0,T]; L^2(0,1)\big) \cap C\big([0,T];
D(A_1)\big),
\end{equation}
and there exists a positive constant $C$ such that
\begin{equation}\label{stima3w}
\begin{aligned}
&\sup_{t \in [0,T]}\left(\|u(t)\|^2_{H^1_a(0,1)}
\right)+ \int_0^{T}
\left(\left\|u_t\right\|^2_{L^2(0,1)}+
\left\|(au_{x})_x\right\|^2_{L^2(0,1)}\right)dt\\
&\le C
\left(\|u_0\|^2_{H^1_a(0,1)} +
\|h\|^2_{L^2(Q_T)}\right).
\end{aligned}\end{equation}
\end{Theorem}

\begin{proof} Observe that $D(A_1)$ is dense in $L^2(0,1)$. In order to show that
$A_1$ is nonpositive and self-adjoint it suffices to prove that
$A_1$ is symmetric, nonpositive and $(I-A_1)(D(A_1))=L^2(0,1)$. Following \cite{fggr}, one can prove that
$A_1$  is symmetric
and
nonpositive. Now, we prove that
$ I-A_1$ is surjective, since the proof is quite different.

First of all, observe that $H^1_a(0,1)$ equipped with the inner
product
$
(u, v)_1:= \int_0^1 (u v + au'v') dx,
$
for any $u, v \in H^1_a(0,1)$, is a Hilbert space. Moreover,
$
H^1_a(0,1) \hookrightarrow L^2(0,1) \hookrightarrow (H^1_a(0,1))^*,
$
where $(H^1_a(0,1))^*$ is the dual space of $H^1_a(0,1)$ with
respect to $L^2(0,1)$ . Now, for
 $f \in L^2(0,1)$,
consider the functional $F: H^1_a(0,1) \rightarrow \R$ defined as
$\displaystyle F(v) := \int_0^1 f v dx$. Clearly, it belongs to $(H^1_a(0,1))^*$. As a
consequence, by the Lax--Milgram Lemma, there exists a unique  $u \in
H^1_a(0,1)$ such that for all $v \in H^1_a(0,1)$
$
(u, v)_1= \int_0^1 f v dx.
$
In particular, since $C_c^\infty(0,1) \subset H^1_a(0,1)$,
the previous equality holds for all $v \in C_c^\infty(0,1)$, i.e.
$
\int_0^1au'v'dx = \int_0^1 (f-u)v dx,$  for all $v
\;\in\; C_c^\infty(0,1).
$
Thus, the distributional derivative of $au'$ is a function in
$L^2(0,1)$, that is $au' \in H^1(0,1)$ (recall that $\sqrt{a}u' \in
L^2(0,1)$) and $(au')' = u-f$ a.e. in $(0,1)$. Then $u \in
H^2_a(0,1)$ and, proceeding as in \cite[Proposition VIII.16]{b}, one can prove that $u'(0)=u'(1)=0$. In fact, by the Gauss Green Identity and
$
(u, v)_1= \int_0^1 f v dx,$  one has that for all $v \in H^1_a(0,1)$
\begin{equation}\label{b1}
\int_0^1 (au')' v dx = [au' v]_{x=0}^{x=1} - \int_0^1 au'v' dx =  [au' v]_{x=0}^{x=1} - \int_0^1 (f-u) vdx.
\end{equation}
In particular, the previous equality holds for all $v \in C_c^\infty(0,1)$. Thus, $[au' v]_{x=0}^{x=1}=0$ for all $v \in C_c^\infty(0,1)$ and
$
(au')' = u-f \; \text{ a. e. in } \;(0,1).
$
Coming back to \eqref{b1}, it becomes
$
[au' v]_{x=0}^{x=1}=0, \; \text{for all}\; v \in H^1_a(0,1).
$
Since $v(0)$ and $v(1)$ are arbitrary and $a$ does not degenerate in $0$ and in $1$, one can conclude that
$
u'(0)=u'(1)=0.
$

Hence $u \in D(A_1)$, and by $(u, v)_1= \int_0^1 f v dx$  and Lemma \ref{green}, we
have
$
\int_0^1(u -(au')'- f)v dx =0.
$
Consequently,
$
u \in D(A_1)$ and $u- A_1u=f.
$

Finally, $A_1$ being a nonpositive self--adjoint operator on a Hilbert
space, it is well known that $(A_1, D(A_1))$ generates a cosine family
and an analytic contractive semigroup of angle $\frac{\pi}{2}$ on
$L^2(0,1)$ (see, e.g., \cite{fggr})).

In the rest of the proof, following \cite[Theorem 2.1]{fm1}, we will prove \eqref{stima2w}--\eqref{stima3w}. First, being
$A_1$ the generator of a strongly continuous semigroup on
$L^2(0,1)$, if $u_0\in L^2(0,1)$, then
the solution $u$ of \eqref{linear} belongs to $C\big([0,T];
L^2(0,1)\big) \cap L^2 \big(0,T;
H^1_a(0,1)\big)$, while, if $u_0\in D(A_1)$ and $h \in W^{1,1}(0,T; L^2(0,1))$, then $u\in
C^1\big([0,T]; L^2(0,1)\big) \cap C\big([0,T];
H^2_a(0,1)\big)$ by \cite[Proposition 3.3]{dp} or \cite[Proposition 4.1.6]{ch}.

Now, we shall prove \eqref{regularity1} and \eqref{stima3w}, from which the last required
regularity property for $u$ will follow by standard linear
arguments.
First, take $u_0\in D(A_1)$ and multiply the equation of \eqref{linear} by $u$; by the
Cauchy--Schwarz inequality we obtain for every $t\in (0,T]$,
\begin{equation}\label{derivo}
\frac{1}{2}\frac{d}{dt}\|u(t)\|^2_{L^2(0,1)}+
\|\sqrt{a}u_x(t)\|^2_{L^2(0,1)}\leq
\frac{1}{2}\|u(t)\|^2_{L^2(0,1)}+\frac{1}{2}
\|h(t)\|^2_{L^2(0,1)},
\end{equation}
from which
\begin{equation}\label{sottoderivo}
\|u(t)\|^2_{L^2(0,1)}\leq
e^T\left(\|u(0)\|^2_{L^2(0,1)}+\|h\|_{L^2(Q_T)}^2
\right)
\end{equation}
for every $t\leq T$. From
\eqref{derivo} and \eqref{sottoderivo} we immediately get
\begin{equation}\label{sottosotto}
\int_0^T\|\sqrt{a}u_x(t)\|^2_{L^2(0,1)}dt\leq
C_T\left(\|u(0)\|^2_{L^2(0,1)}+\|h\|_{L^2(Q_T)}^2
\right)
\end{equation}
for every $t\leq T$ and some universal constant $C_T>0$. Thus, by
\eqref{sottoderivo} and \eqref{sottosotto}, \eqref{stima2w} follows
if $u_0\in D(A_1)$. Since $D(A_1)$ is dense in $L^2(0,1)$,
the same inequality holds if $u_0\in L^2(0,1)$.

Now, we multiply the equation by $-(au_x)_x$, we integrate on $(0,1)$
and we easily get
$
\displaystyle \frac{d}{dt}\|\sqrt{a}u_x(t)\|^2_{L^2(0,1)}+\|(au_{x})_x(t)\|^2_{L^2(0,1)}\leq
\|h(t)\|_{L^2(0,1)}^2
$
for every $t$, so that, as before, we find $C_T'>0$ such that
\begin{equation}\label{mah}
\|\sqrt{a}u_x(t)\|_{L^2(0,1)}+\int_0^T\|(au_{x})_x(t)\|^2_{L^2(0,1)}dt
\leq
C_T'\left(\|\sqrt{a}u_x(0)\|_{L^2(0,1)}+\|h\|_{L^2(Q_T)}^2\right)
\end{equation}
for every $t\!\leq \!T$.\!
Finally, from $u_t\!\!=\!\!(au_{x})_x+h$, squaring and integrating,
we find
$
\!\int_0^T\!\!\!\|u_t(t)\|_{L^2(0,1)}^2\!\leq\!
C\left(\!\int_0^T\|(au_{x})_x\|^2_{L^2(0,1)}\!\!+\!\!\|h\|_{L^2(Q_T)}^2
\!\right),
$
and together with \eqref{mah} we find
\begin{equation}\label{allafine}
\int_0^T\|u_t(t)\|_{L^2(0,1)}^2\leq
C\left(\|\sqrt{a}u_x(0)\|_{L^2(0,1)}+\|h\|_{L^2(Q_T)}^2\right).
\end{equation}

In conclusion, \eqref{derivo}, \eqref{sottoderivo}, \eqref{mah} and \eqref{allafine}
give \eqref{stima2w} and \eqref{stima3w}.  Clearly,
\eqref{regularity1} and \eqref{stima3w} hold also if
$u_0\in H^1_a(0,1)$,  since $H^2_a(0,1)$
is dense in $ H^1_a(0,1)$.
\end{proof}

\subsection{Strongly degenerate operator}
In this subsection we assume that the operator is strongly
degenerate. Following \cite{acf},  we introduce the weighted space
\[
\begin{aligned}
H^1_a(0,1):=\{ u \in L^2(0,1) \ \mid  \,&u \text{ locally absolutely
continuous in } [0,x_0) \cup (x_0,1] \\ &\text{ and } \sqrt{a} u' \in  L^2(0,1)\}
\end{aligned}
\]
with the norm given in \eqref{norm}. Define the
operator $A_1$ by
$D(A_1)= \{ u \in H^2_a(0,1)| \, u'(0)=u'(1)=0\},$
 and, for any $ \,u \in D(A_1),$
$A_1u:=(au')',
$
where $(H^2_a(0,1), \|\cdot\|_{H^2_a(0,1)})$ is defined as before. Since in this case a
function $u \in H^2_a(0,1)$ is locally absolutely continuous in
$[0,1]\setminus\{x_0\}$ and not necessarily absolutely continuous in
$[0,1]$ as for the weakly degenerate case, equality \eqref{greenformula}
is not true a priori. Thus, as in \cite{fggr}, we have to prove again the
formula of integration by parts. To do this, an idea is to characterize the domain of $A_1$.
The next results hold:
\begin{Proposition}\label{characterization}
Let
\[
\begin{aligned} X:=\{ u \in L^2(0,1)\,| \ &u \text{ locally
absolutely continuous in } [0,1]\setminus \{x_0\},  \\ &\sqrt{a}u'
\in L^2(0,1), au \text{ is continuous at } x_0 \;\text{and }
(au)(x_0)=0\}.
\end{aligned}
\]
Then $
H^1_a(0,1)=X.
$
\end{Proposition}
\begin{proof} Obviously, $X \subseteq H^1_a$.
Now we take $u \in
H^1_a$ and we prove that $au$ is continuous at $x_0$ and $(au)(x_0)=0$, that is $u \in X$. Toward
this end, observe that  since $a \in W^{1, \infty}(0,1)$, $(au)' = a'
u+ au' \in L^2(0,1)$.  Thus, for $x <x_0$, one has
$
au(x)= (au)(0) + \int_0^x(au)'(t)dt$
(observe that $(au)(0) \in \R$).
 This
implies that there exists $\lim_{x \rightarrow x_0^-}(au)(x)= (au)(x_0)
= (au)(0)+\int_0^{x_0}(au)'(t)dt= L \in \R$.
 As in \cite[Proposition 2.3]{fggr}, one can prove that $L=0$. Analogously, $\lim_{x \rightarrow x_0^+}(au)(x)=(au)(x_0)
=0$. Thus
$(au)(x_0)=0$.
\end{proof}

Using the previous result, one can prove the following
characterization:
\begin{Proposition}\label{domain}
Let
\[
\begin{aligned} D:=\{ u \in L^2(0,1)\,| \ &u \text{ locally
absolutely continuous in } [0,1]\setminus \{x_0\},  au \in
H^1(0,1), \\ & a u' \in H^1(0,1),  au \text{ is continuous at } x_0 \;\\ &\text{and } (au)(x_0)=(au')(x_0)=u'(0)=u'(1)=0
\}.
\end{aligned}
\]
Then
$
 D(A_1)=D.
$
\end{Proposition}

\begin{proof} Let us prove that
$D=D(A_1)$.

{$\boldsymbol{D\subseteq D(A_1):}$} It is a simple adaptation of the proof of \cite[Proposition 2.4]{fggr} to which we refer. We underline the fact that here we
use the boundary conditions $u'(0)=u'(1)=0$.

{$\boldsymbol {D(A_1) \subseteq D:}$} As in the
proof of Proposition \ref{characterization}, we can prove that $au, (au)' \in L^2(0,1)$, thus $au \in H^1(0,1)$. Moreover, by Proposition
\ref{characterization},
$(au)(x_0)=0$. Thus, it is sufficient to prove that $(au')(x_0)=0$. This follows as in \cite[Proposition 2.4]{fggr}.
\end{proof}

We point out the fact that to prove the previous characterization the condition $\displaystyle\frac{1}{a}
\not \in L^1(0,1)$ is crucial.
Clearly this condition is not satisfied if the
operator is weakly degenerate. Indeed, in \cite[Lemma 2.1]{fm} it is proved that if the operator is weakly degenerate, then $\displaystyle \frac{1}{a}\in L^1(0,1)$; on the other hand, if it is strongly degenerate then $ \displaystyle \frac{1}{\sqrt{a}}\in L^1(0,1)$, while $\displaystyle \frac{1}{a}\not \in L^1(0,1)$.

\vspace{0,5cm}
Proceeding as in \cite[Lemma 2.6]{fggr} and using the previous characterization, we can prove the
formula of integration by parts \eqref{greenformula} also in the strongly degenerate case.
Thus, the analogue of Theorem \ref{prop} holds.

\section{Well posedness in the non divergence case}\label{non_divergence}
Now, we consider the operator $A_2u=au''$ in the weakly and in the strongly degenerate cases and, as in \cite[Chapter 2]{fm1}, we consider the following Hilbert
spaces:
\[L^2_{\frac{1}{a}}(0,1) :=\left\{ u \in L^2(0,1) \
\mid \int_0^1 \frac{u^2}{a} dx <\infty \right\},
\]
\[
H^1_{\frac{1}{a}}(0,1) :=L^2_{\frac{1}{a}}(0,1)\cap H^1(0,1) \quad \text{and} \quad H^2_{\frac{1}{a}}(0,1) :=\Big\{ u \in H^1_{\frac{1}{a}}(0,1) \,
\big| \, u'\in H^1(0,1)\Big\},\]
endowed with the associated norms
$ \displaystyle \|u\|_{L^2_{\frac{1}{a}}(0,1)}^2:= \int_0^1
\frac{u^2}{a} dx,$ $\forall\, u\in L^2_{\frac{1}{a}}(0,1),$
$
\|u\|^2_{H^1_{\frac{1}{a}}}:=\|u\|_{L^2_{\frac{1}{a}}(0,1)}^2 +
\|u'\|^2_{L^2(0,1)},$ $\forall\, u\in H^1_{\frac{1}{a}}(0,1)$
 and
$
\|u\|_{H^2_{\frac{1}{a}}(0,1)}^2 :=\|u\|_{H^1_{\frac{1}{a}}(0,1)}^2 +
\|au''\|^2_{L^2_{\frac{1}{a}}(0,1)},$ $\forall\,u\in
H^2_{\frac{1}{a}}(0,1),
$ respectively.
Indeed, it is a trivial fact that, if $u'\in H^1(0,1)$, then $au''
\in L^2_{\frac{1}{a}}(0,1)$, so that the norm for
$H^2_{\frac{1}{a}}(0,1)$ is well defined and we can also write in a
more appealing way
\[
H^2_{\frac{1}{a}}(0,1) :=\Big\{ u \in H^1_{\frac{1}{a}}(0,1) \,
\big| \, u'\in H^1(0,1) \mbox{ and }au'' \in
L^2_{\frac{1}{a}}(0,1)\Big\}.
\]
Using the
previous spaces, we define the operator $A_2$ by
$
 D(A_2)= \{ u \in H^2_{\frac{1}{a}}(0,1)| \, u'(0)=u'(1)=0\}$
 and, for any
$u \in D(A_2), $
$A_2u:=au''.
$

Proceeding as in \cite[Corollary 3.1]{fggr}, one can prove the following characterization:
\begin{Corollary}
If the operator is weakly degenerate, then the spaces $H^1_{\frac{1}{a}}(0,1)$ and $H^1(0,1)$ coincide
algebraically. Moreover the two norms are equivalent.
\end{Corollary}

In every case
$C_c^\infty(0,1)$ is contained in $H^1_{\frac{1}{a}}(0,1)$.

As for the divergence form,
a crucial tool is the following
formula of integration by parts:
\begin{Lemma}\label{green1}
For all $(u,v)\in D(A_2)\times
H^1_{\frac{1}{a}}(0,1)$ one has
\begin{equation}\label{greenformula1}
\int_0^1u'' v\, dx= - \int_0^1 u'v'\, dx.
\end{equation}
\end{Lemma}
\begin{proof}
It is trivial, since $u'(0)\!\!=\!\!u'(1)\!\!=\!\!0$ and both $u'\!\!\in\! H^1(0,1)$ and $v\!\in\! H^1(0,1)$.
\end{proof}

We also recall the following definition:
\begin{Definition}
Assume that $u_0 \in L^2_{\frac{1}{a}}(0,1)$ and $h\in
L^2_{\frac{1}{a}}(Q_T):= L^2(0,T; L^2_{\frac{1}{a}}(0,1))$. A function $u$ is said to be a weak
solution of \eqref{linear} with $A= A_2$ if
\[u \in C([0, T]; L^2_{\frac{1}{a}}(0,1)) \cap L^2(0, T;
H^1_{\frac{1}{a}}(0,1))\] and satisfies
\[
\begin{aligned}
&\int_0^1 \frac{ u(T,x)\varphi(T,x)}{a(x)} dx - \int_0^1
\frac{u_0(x) \varphi(0,x)}{a(x)} dx -\int_{Q_T}
\frac{\varphi_t (t,x)u(t,x)}{a(x)}dxdt =
\\&- \int_{Q_T} u_x(t,x)
\varphi_x(t,x) dxdt + \int_{Q_T} h(t,x) \frac{\varphi(t,x)
}{a(x)}dx dt
\end{aligned}
\]
for all $\varphi \in H^1(0, T; L^2_{\frac{1}{a}}(0,1)) \cap L^2(0,
T; H^1_{\frac{1}{a}}(0,1))$.
\end{Definition}
As a consequence of the previous lemma one has the next proposition,
whose proof is similar to the proof of Theorem \ref{prop}.
\begin{Theorem}\label{theorem_nondivergence}
The operator $A_2:D(A_2)\to L^2_{\frac{1}{a}}(0, 1)$ is self--adjoint,
nonpositive on $L^2_{\frac{1}{a}}(0,1)$ and it generates an analytic
contraction semigroup of angle $\pi/2$. Therefore, for all $h
\in L^2_{{\frac{1}{a}}}(Q_T)$ and $u_0 \in
L^2_{{\frac{1}{a}}}(0,1)$, there exists a unique solution
\[u \in
C\big([0,T]; L^2_{\frac{1}{a}}(0,1)\big)\cap L^2 \big(0,T;
H^1_{\frac{1}{a}}(0,1)\big)
\] of \eqref{linear} such that
\begin{equation}\label{stima2}
\sup_{t \in [0,T]}
\|u(t)\|^2_{L^2_{\frac{1}{a}}(0,1)}+\int_0^T\|u(t)\|^2_{H^1_{\frac{1}{a}}
(0,1)} dt \le
C_T\left(\|u_0\|^2_{L^2_{\frac{1}{a}}(0,1)}+\|h\|^2_{L^2_{\frac{1}{a}}(Q_T)}\right),
\end{equation}
for some positive constant $C_T$. Moreover, if $h \in W^{1,1}(0,T; L^2_{\frac{1}{a}}(0,1))$ and $u_0 \in
H^1_{\frac{1}{a}}(0,1)$, then
\begin{equation}\label{regularity1'}
u\in C^1\big([0,T]; L^2_{\frac{1}{a}}(0,1)\big) \cap C\big([0,T];
D(A_2)\big),
\end{equation}
and there exists a positive constant $C$ such that
\begin{equation}\label{stima3}
\begin{aligned}
\sup_{t \in [0,T]}\left(\|u(t)\|^2_{H^1_{\frac{1}{a}}(0,1)}
\right)&+ \int_0^{T}
\left(\left\|u_t\right\|^2_{L^2_{\frac{1}{a}}(0,1)} +
\left\|au_{xx}\right\|^2_{L^2_{\frac{1}{a}}(0,1)}\right)dt\\&\le C
\left(\|u_0\|^2_{H^1_{\frac{1}{a}}(0,1)} +
\|h\|^2_{L^2(Q_T)}\right).
\end{aligned}\end{equation}
\end{Theorem}
\begin{proof}In the (SD) case for the existence and the regularity parts, we can proceed as in \cite[Theorem 2.2]{fm1}, to which we refer.
In the (WD) case, we proceed as in Theorem \ref{prop}: first, observe that $D(A_2)$ is dense in $L^2_{\frac{1}{a}}(0,1)$. Then, using Lemma \ref{green1}, one has that
$A_2$ is symmetric and nonpositive.
Finally, let us show that $ I-A_2$ is surjective. First of all,
observe that $H^1_{\frac{1}{a}}(0,1)$ is equipped with the natural
inner product
$
(u, v)_1:=\displaystyle \int_0^1 \left(\frac{u v}{a} + u'v'\right) dx
$
for any $u, v \in H^1_{\frac{1}{a}}(0,1)$. Moreover, it is clear
that
$
H^1_{\frac{1}{a}}(0,1) \hookrightarrow L^2_{\frac{1}{a}}(0,1)
\hookrightarrow (H^1_{\frac{1}{a}}(0,1))^*,
$
where $(H^1_{\frac{1}{a}}(0,1))^*$ is the dual space of
$H^1_{\frac{1}{a}}(0,1)$ with respect to $L^2_{\frac{1}{a}}(0,1)$.
Now, if $f \in L^2_{\frac{1}{a}}(0,1)$, consider the functional $F:
H^1_{\frac{1}{a}}(0,1) \to \R$ defined as $ F(v) := \displaystyle\int_0^1
\frac{fv}{a} dx$. Clearly it belongs to $(H^1_{\frac{1}{a}}(0,1))^*$. As a
consequence, by the Lax--Milgram Lemma, there exists a unique $u \in
H^1_{\frac{1}{a}}(0,1)$ such that for all $v \in
H^1_{\frac{1}{a}}(0,1)$,
$
(u, v)_1= \ds \int_0^1 \frac{f v}{a} dx.
$
In particular, since $C^\infty_c(0,1) \subset
H^1_{\frac{1}{a}}(0,1)$, the previous equality holds for all $v \in
C^\infty_c(0,1)$, i.e.
$
\displaystyle \int_0^1u'v'dx = \int_0^1 \frac{(f-u)}{a}v\, dx,$ for
every $v \in C^\infty_c(0,1).
$
Thus, the distributional derivative of $u'$ is a function in
$L^2_{\frac{1}{a}}(0,1)\subset L^2(0,1)$, hence it is easy to see
that $au''\in L^2_{\frac{1}{a}}(0,1)$.
Thus $u \in H^2_{\frac{1}{a}}(0,1)$.
Proceeding as in Theorem \ref{prop}, one can prove that $u'(0)=u'(1)=0$. In fact, by the Gauss Green Identity and $
(u, v)_1= \displaystyle\int_0^1 \frac{f v}{a} dx$, one has that for all
$v \in H^1_{\frac{1}{a}}(0,1)$,
\begin{equation}\label{b1*}
\int_0^1 u'' v dx = [u' v]_{x=0}^{x=1} - \int_0^1 u'v' dx =  [u' v]_{x=0}^{x=1} - \int_0^1 \frac{(f-u)}{a} vdx.
\end{equation}
In particular, the previous equality holds for all $v \in C_c^\infty(0,1)$. Thus, $[u' v]_{x=0}^{x=1}=0$ for all $v \in C_c^\infty(0,1)$ and
$
u''= \displaystyle \frac{(u-f)}{a}\; \text{ a. e. in } \;(0,1).
$ Coming back to \eqref{b1*}, it becomes
$
[u' v]_{x=0}^{x=1}=0, \; \text{for all}\; v \in H^1_{\frac{1}{a}}(0,1).
$
Again, one can conclude that
$
u'(0)=u'(1)=0.
$
Thus $u \in D(A_2)$, and by $(u, v)_1= \displaystyle\int_0^1 \frac{f v}{a} dx$ and Lemma \ref{green1}, we
have
$
\displaystyle\int_0^1\left(\frac{u-f }{a}-u''\right)v dx =0.
$
Consequently,
$
u \in D(A_2)$ and $u- A_2u=f.
$
As in Theorem \ref{prop}, one can conclude
 that $(A_2, D(A_2))$ generates a cosine family
and an analytic contractive semigroup of angle $\ds\frac{\pi}{2}$ on
$L^2_{\frac{1}{a}}(0,1)$.
The rest of the theorem follows as in \cite[Theorem 2.2]{fm1}.
\end{proof}

\subsection{Characterizations in the strongly degenerate case}
In this subsection we will concentrate, as in \cite{fggr}, on the strongly degenerate case and  we will characterize the
 spaces
$H^1_{\frac{1}{a}}(0,1)$ and $H^2_{\frac{1}{a}}(0,1)$. We point out
the fact that in non divergence form, the characterization of the
domain of the operator is not important to prove the formula of integration by parts
as in divergence form.

First of all observe that, as in \cite[Lemma 2.1]{fm}, one can prove that
$\displaystyle\frac{|x-x_0|^2}{a(x)} \le C,$
for all $x
\in [0,1]\setminus\{x_0\}$,
where $C:= \displaystyle \max\left\{\frac{(x_0)^2}{a(0)}, \frac{(1-x_0)^2}{a(1)} \right\}$.
The following characterization holds:
\begin{Proposition}\label{domain3} Let
$
X:=\{ u \in H^1_{\frac{1}{a}}(0,1) \ \mid u(x_0)=0
\}.
$
If $A_2$ is strongly degenerate, then
$
H^1_{\frac{1}{a}}(0,1)=X
$
and, for all $u\in X$, $\|u\|_{H^1_{\frac{1}{a}(0,1)}}$ is
equivalent to
$\left(\int_0^1(u')^2dx\right)^{\frac{1}{2}}$.
\end{Proposition}

The proof of the previous proposition is a simple adaptation of the proof of \cite[Proposition 3.6]{fggr}, to which we refer.
 An
immediate consequence of Proposition \ref{domain3} is the following result.
\begin{Proposition}\label{domain4}
Let
\[
\begin{aligned}
D:=\{ u \in H^1_{\frac{1}{a}}(0,1) \ \mid  \;au'' \in
L^2_{\frac{1}{a}}(0,1),\; u' \in H^1(0,1) \text{ and }
u(x_0)=(au')(x_0)=0\}.
\end{aligned}
\]

 If $A_2$ is strongly degenerate, then
$
H^2_{\frac{1}{a}}(0,1)=D.
$
\end{Proposition}
\begin{proof}
Obviously, $D \subseteq H^2_{\frac{1}{a}}(0,1)$. Now, we take $u \in
H^2_{\frac{1}{a}}(0,1)$ and we prove that $u \in D$.  By Proposition
\ref{domain3}, $u(x_0)=0$. Thus, it is sufficient to prove that $(au')(x_0)=0$. Since $u' \in H^1(0,1)$ and $a
\in W^{1, \infty}(0,1)$, then $au' \in C[0,1]$ and $\sqrt{a} u' \in L^2(0,1)$. This implies that there exists
$\lim_{x \rightarrow x_0}(au')(x) = (au')(x_0)=L \in \R$. Proceeding as in the proof of \cite[Proposition 3.6]{fggr}, one can prove that $L=0$, that is
$(au')(x_0)=0$.
\end{proof}

\section{Carleman estimate for degenerate parabolic
problems: the divergence case}\label{Carleman estimate}

In this section we prove an interesting estimate of Carleman type for the
adjoint problem of \eqref{linear} in divergence form
\[
\begin{cases}
v_t + \left(av_x \right) _x =h, & (t,x) \in Q_T,\\
v_x(t,0)=v_x(t,1)=0, &  t \in (0,T),\\
v(T,x)= v_T(x)\in L^2(0,1),
\end{cases}
\]
where $T>0$ is given. As it is well known, to prove Carleman estimates the final datum is irrelevant, only the equation and the boundary conditions are important. For this reason we can consider only the problem
\begin{equation}\label{1}
\begin{cases}
v_t + \left(av_x \right) _x =h, & (t,x) \in Q_T,\\
v_x(t,0)=v_x(t,1)=0, &  t \in (0,T).
\end{cases}
\end{equation}
Here we assume that $h \in L^2(Q_T)$ and on $a$ we make the following assumptions:
\begin{Assumptions}\label{Ass03}
The function $a$ is such that
\begin{enumerate}
\item the operator $A_1$ is weakly or strongly degenerate;
\item in the weakly degenerate case there exist $B_1\in (0,x_0)$, $B_2\in (1,2-x_0)$,
two functions $\fg \in L^\infty_{\rm loc}((-x_0, 2-x_0)\setminus \{x_0\})$, $\fh (\cdot, B_i) \in W^{1,\infty}_{\rm loc}((-x_0, 2-x_0)\setminus \{x_0\})$ and
two strictly positive constants $\fg_0$, $\fh_0$ such that $\fg(x) \ge \fg_0$ and
\begin{equation}\label{rieccola}
-\frac{\tilde a'(x)}{2\sqrt{\tilde a(x)}}\left(\int_x^{B_i}\fg(t) dt + \fh_0 \right)+ \sqrt{\tilde a(x)}\fg(x) =\fh(x,B_i)\quad \text{for a.e.} \; x \in (-x_0, 2-x_0)
\end{equation}
with $i=1,2$, $-x_0< x<B_1$ or $x_0 <x <B_2$, and
\begin{equation}\label{tildea}
\tilde a(x):= \begin{cases}
a(2-x),& x \in [1,2],\\
a(x), & x \in [0,1],\\
a(-x), & x \in [-1,0];
\end{cases}
\end{equation}
 \item
if $A_1$ is strongly degenerate and $K > \displaystyle
\frac{4}{3}$, then there exists a constant $\vartheta \in
(0, K]$ such that the function
\begin{equation}\label{dainfinito}
\begin{array}{ll}
x \mapsto \dfrac{a(x)}{|x-x_0|^{\vartheta}} &
\begin{cases}
& \mbox{ is nonincreasing on the left of $x=x_0$,}\\
& \mbox{ is nondecreasing on the right of $x=x_0$.}\\
\end{cases}
\end{array}
\end{equation}
In addition, when $K > \displaystyle\frac{3}{2}$, the previous map is bounded below away from $0$
and there exists a constant $\Sigma>0$ such that
$
|a'(x)|\leq \Sigma |x-x_0|^{2\vartheta-3} \mbox{ for a.e. }x\in
[0,1].
$

Here $K$ is the constant that appears in Definition \ref{def2}.

\end{enumerate}
\end{Assumptions}

\begin{Remark}\label{remark5} The additional requirements when $K>3/2$
are technical
ones and are introduced in \cite[Hypothesis 4.1]{fm1} to guarantee the convergence of some
integrals for this sub-case (see \cite[Appendix]{fm1}). Of course,
the prototype $a(x)=|x-x_0|^K$ satisfies such a condition with
$\vartheta=K$.
\end{Remark}

\begin{Remark}\label{remspostata}
Since we require identities \eqref{rieccola}  far from $x_0$, once $a$ is given, it is easy to find $\fg,\fh,\fg_0$ and $\fh_0$ with the desired properties. For example, if $a(x):= |x-x_0|^\alpha, \alpha \in (0,1)$,  in \eqref{rieccola} we can take $\fg_0=\min \left\{1, \frac{\alpha(B_1+1)+x_0}{x_0}, \frac{\alpha B_2+ 1-x_0}{1-x_0}\right\}$, $\fh_0=1 $, $
\fg(x) = 1$ in $[0,x_0]\cup[1,2-x_0)$, while 
$
\fg(x) =
\frac{\alpha(B_1+1)+x_0}{x_0}=:L_1$ in $(-x_0, 0)$ and 
$
\fg(x) =\frac{\alpha B_2+ 1-x_0}{1-x_0}=:L_2$ in $(x_0, 1)$.
On the other hand  
\[\fh (x,B_1)= \begin{cases}\ds (x_0-x)^{\frac{\alpha}{2}-1}\left[ \frac{\alpha}{2} (B_1+1-x) + (x_0-x)\right], & x \in [0, B_1),\\
(x+x_0)^{\frac{\alpha}{2}-1}\left[\ds-\frac{\alpha}{2} (-L_1x+B_1+1) + L_1(x+x_0)\right], & x  \in (-x_0, 0)
\end{cases}
\]
and
\[\fh (x,B_2)= \begin{cases}
(x-x_0)^{\frac{\alpha}{2}-1} \left[\ds-\frac{\alpha}{2}(L_2(1-x) + B_2) + L_2(x-x_0)\right], & x  \in (x_0, 1)\\
\ds (2-x-x_0)^{\frac{\alpha}{2}-1}\left[ \frac{\alpha}{2} (B_2+1-x) + (2-x-x_0)\right], & x \in [1, B_2),\\
\end{cases}
\]
 Clearly, $\fg \in L^\infty_{\rm loc}((-x_0, 2-x_0)\setminus \{x_0\})$ and $\fh \in W^{1,\infty}_{\rm loc}((-x_0, 2-x_0)\setminus \{x_0\})$.
 
 Obviously, with $W^{1,\infty}_{\rm loc}((-x_0, 2-x_0)\setminus \{x_0\})$ we denote the space of functions belonging to $W^{1,\infty}((-x_0, 2-x_0))$ far away from $\{x_0\}$.

\end{Remark}

As in \cite{fm} or in \cite[Chapter 4]{fm1}, let us introduce the functions
$
\varphi(t,x):
=\Theta(t)\psi(x),
$
where
\begin{equation}\label{c_1}
\Theta(t) := \frac{1}{[t(T-t)]^4} \quad \text{and} \quad
\psi(x) := c_1\left[\int_{x_0}^x \frac{y-x_0}{a(y)}dy- c_2\right],
\end{equation}
with $c_2> \displaystyle \max\left\{\frac{(1-x_0)^2}{a(1)(2-K)},
\frac{x_0^2}{a(0)(2-K)}\right\}$ and $c_1>0$.
 Observe that $\Theta (t)
\rightarrow + \infty \, \text{ as } t \rightarrow 0^+, T^-$ and by
\cite[Lemma 2.1]{fm}, we have that
$
-c_1c_2 \le \psi(x)<0$. Therefore, define, for $A<B$,
\begin{equation}\label{c_1nd}
\rho_{A,B}(x)\!:=\!\!\begin{cases} \displaystyle - r\left[\int_{A}^x\!
\frac{1}{\sqrt{\tilde a(t)}} \int_t^{B}
\fg(s) dsdt\! + \!\!\int_{A}^x \!\frac{\fh_0}{\sqrt{\tilde a(t)}}dt\right] -\mathfrak{c},\!\!\! &\text{in the weakly degenerate case,}\\
\displaystyle  e^{r\zeta(x)}-\mathfrak{c},\!\!\! &\text{in the strongly degenerate case,}\end{cases}
\end{equation}
where
 \[
\zeta(x)=\mathfrak{d}\int_x^{B}\frac{1}{\tilde a(t)}dt,
\]
being $\fd=\|\tilde a'\|_{L^\infty(A,B)}$,  $r>0$ and $\mathfrak{c}>0$ such that $\displaystyle
\max_{[A,B]} \rho_{A,B}<0$.

Our main result is thus the following:
\begin{Theorem}\label{Cor1}
Assume Hypothesis  $\ref{Ass03}$.
 Let $\omega\subset \subset (0,1)$ be an open interval containing $x_0$, $B_1$ and $2-B_2$, or let $\omega= \omega_1\cup \omega_2$, where $\omega_i=(\lambda_i,\beta_i) \subset (0,1), \, i=1,2,$
$\beta_1\leq B_1$ and $2-B_2\le \lambda_2$. Then, there exist two positive constants $C$ and $s_0$ (depending on $\lambda$) such that every
solution  $v$
of \eqref{1} in
$
\mathcal{V}:=L^2\big(0, T; D(A_1)\big) \cap H^1\big(0,
T;H^1_a(0,1)\big)
$
satisfies, for all $s \ge s_0$,
\begin{equation}\label{carleman1}
\begin{aligned}
&\int_0^T\int_0^1 \left(s\Theta a(v_x)^2 + s^3 \Theta^3
\frac{(x-x_0)^2}{a} v^2\right)e^{2s\varphi}dxdt \le C  \int_0^T \int_\omega v^2dxdt\\
& + C\left(\int_0^T\int_0^1 h^{2}e^{2s\varphi}dxdt+\int_0^T \int_0^{\beta_1}
 h^2 e^{2s \Phi_1(t,-x)}dxdt+ \int_0^T \int_{\lambda_2}^1 h^2e^{2s \Phi_2(t,x)}dxdt\right).
 \end{aligned}
\end{equation}
Here
$\Phi_{1}(t,x): =\Theta(t)\rho_{-\beta_1, \beta_1}(x)$ and  $\Phi_{2}(t,x): =\Theta(t)\rho_{\lambda_2, 2-\lambda_2}(x)$.
\end{Theorem}

\begin{Remark}\label{oss_carleman}
Observe that an inequality analogous to \eqref{carleman1} in the non degenerate case is proved in \cite{LRL}, where the authors show that
\begin{equation}\label{lrl}
\int_0^T\int_0^1 \left(s\Theta (v_x)^2 + s^3 \Theta^3
v^2\right)e^{2s\varphi}dxdt\le C\left(\int_0^T\int_0^1 h^{2}e^{2s\varphi}dxdt+ s^3\int_0^T \int_\omega \Theta^3 v^2 e^{2s\varphi} \right),
\end{equation}
for a different weight function $\varphi$ and for a fixed subset $\omega$ compactly contained in $(0,1)$.
We underline that we don't have $s^3\Theta^3$ in
the term $ \int_0^T\int_\omega v^2e^{2s\varphi}dxdt$ and we cannot  estimated  such an integral by
\[s^3\int_0^T\int_0^1 \Theta^3
\frac{(x-x_0)^2}{a} v^2e^{2s\varphi}dxdt\]
due to the degeneracy term, and so \eqref{carleman1} is a good alternative of \eqref{lrl}.
 \end{Remark}

In order to prove the previous theorem the following Carleman estimate given in \cite[Theorem 4.1]{fm1} is crucial:
\begin{Theorem}\label{Cor1fm1}
Assume Hypothesis  $\ref{Ass03}$. Then, there exist
two positive constants $C$ and $s_0$ such that every solution $v \in L^2\big(0, T; {\cal H}^2_a(0,1)\big) \cap H^1\big(0,
T;{\cal H}^1_a(0,1)\big)$ of
\[
\begin{cases}
v_t + (av_x)_x =h, & (t,x) \in (0,T) \times (0,1),\\
v(t,0)=v(t,1)=0, &  t \in (0,T)\\
\end{cases}
\]
satisfies, for all $s \ge s_0$,
\[
\begin{aligned}
&\int_{Q_T} \left(s\Theta a(v_x)^2 + s^3 \Theta^3
\frac{(x-x_0)^2}{a} v^2\right)e^{2s\varphi}dxdt\\
&\le C\left(\int_{Q_T} h^{2}e^{2s\varphi}dxdt +
sc_1\int_0^T\left[a\Theta e^{2s \varphi}(x-x_0)(v_x)^2
dt\right]_{x=0}^{x=1}\right),
\end{aligned}
\]
where $c_{1}$ is the constant introduced in \eqref{c_1}. Here
\[
\begin{aligned}
{\cal H}^1_a(0,1):=\big\{&u \text{ is absolutely continuous in }
[0,1],
\\ & \sqrt{a} u' \in  L^2(0,1) \text{ and } u(0)=u(1)=0
\big\},
\end{aligned}
\]
in the weakly degenerate case and
\[
\begin{aligned}
{\cal H}^1_a(0,1):=\big\{ u \in L^2(0,1) \ \mid  \,&u \text{ locally
absolutely continuous in } [0,x_0) \cup (x_0,1], \\ & \sqrt{a} u'
\in L^2(0,1) \text{ and } u(0)= u(1)=0 \big\}
\end{aligned}
\]
in the strong one. In any case
\[
\label{Ha2} \qquad {\cal H}^2_a(0,1) :=  \big\{ u \in {\cal
H}^1_a(0,1) |\,au' \in H^1(0,1)\big\}.
\]

\end{Theorem}
We underline the fact that in \cite{fm1} the previous theorem is proved in the weakly degenerate case under weaker assumptions.
\begin{proof}[Proof of Theorem \ref{Cor1}]
To prove the statement we use a technique based on cut off functions. First, assume that $x_0,B_1, 2-B_2\in \omega$. Then, we can fix two subintervals
$\omega_1=(\lambda_1,\beta_1)\subset \subset(0, x_0),
\omega_2=(\lambda_2,\beta_2) \subset\subset (x_0,1)$, with $\beta_1=B_1$ and $\lambda_2=2-B_2$, and four points $\bar \lambda_i, \bar \beta_i  \in (\lambda_i, \beta_i)$, $i=1,2$, with $\bar \lambda_ i  < \bar \beta_i$. Consider a smooth
function $\xi:[-1,2]\to[0,1]$ such that
\begin{equation}\label{xi1}
\xi(x)=\begin{cases} 0&x\in [-1,-\bar\beta_1]\cup[\bar \beta_1, \bar\lambda_2] \cup[2-\bar \lambda_2, 2],\\
1 & x\in[-\tilde \lambda_1,\tilde \lambda_1] \cup [\tilde\lambda_2, 2-\tilde\lambda_2],\end{cases}
\end{equation}
where $\tilde \lambda_i=(\bar \lambda_i+\bar \beta_i)/2$, $i=1,2$. 
Now, we consider
\begin{equation}\label{W}
W(t,x):= \begin{cases}v(t,-x), & x \in [-1,0],\\
v(t,x), & x \in [0,1],\\
v(t, 2-x), &x \in [1,2],
\end{cases}
\end{equation}
where $v$ solves \eqref{1}.
Thus $W$ satisfies the following problem
\begin{equation}\label{PW}
\begin{cases}
W_t + \left(\tilde a W_x \right) _x =\tilde h, & (t,x) \in (0,T) \times (-1,2),\\
W_x(t,-1)=W_x(t,2)=0, &  t \in (0,T),\\
\end{cases}
\end{equation}
being \begin{equation}\label{tildeh}
\tilde h(t,x) := \begin{cases}
h(t, 2-x), &x \in [1,2],\\
h(t,x), & x \in [0,1],\\
h(t,-x), & x \in [-1,0].
\end{cases}
\end{equation}
Observe that $\tilde a$ belongs to $W^{1,1}(-1,2)$ in the weakly degenerate case and to $W^{1,\infty}(-1,2)$ in the strongly degenerate one.
Now, set $Z:= \xi W$. Then $Z$ solves
\begin{equation}\label{p1}
\begin{cases}
Z_t + \left(\tilde a Z_x \right) _x =H, & (t,x) \in (0,T) \times (-\beta_1, \beta_1),\\
Z(t,-\beta_1)=Z(t,\beta_1), & t \in (0,T),\\
\end{cases}
\end{equation}
and
\begin{equation}\label{p1'}
\begin{cases}
Z_t + \left(\tilde a Z_x \right) _x =H, & (t,x) \in (0,T) \times (\lambda_2, 2-\lambda_2),\\
Z(t,\lambda_2)= Z(t, 2-\lambda_2)=0, &  t \in (0,T),\\
\end{cases}
\end{equation}
 with $H:=\xi \tilde h+ ( \tilde a \xi_x W)_x + \tilde
a\xi_x W_x$. Observe that $Z_x(t, -\beta_1)= Z_x(t, \beta_1)=Z_x(t, \lambda_2)=Z_x(t,2-\lambda_2)=0$ and, by the assumption on $a$ and the fact that $\xi_x $ is supported in $[-\bar\beta_1,-\tilde\lambda_1]\cup[\tilde\lambda_1, \bar \beta_1]\cup[\bar \lambda_2, \tilde \lambda_2]\cup[2-\tilde \lambda_2, 2 -\bar \lambda_2]$, $H\in L^2((0,T) \times I)$, where $I:= (-\beta_1, \beta_1)\cup(\lambda_2, 2-\lambda_2)$.

Thus, we can apply the
analogue of \cite[Theorem 3.1]{fm1} on $(-\beta_1,\beta_1)$ in place of
$(A, B)$ and with weight $\Phi_1$, obtaining that there exist
two positive constants $C$ and $s_0$ ($s_0$ sufficiently large), such that $Z$
satisfies, for all $s \ge s_0$,
\[
\int_0^T\int_{-\beta_1}^{\beta_1} \left(s\Theta (Z_x)^2 + s^3 \Theta^3
Z^2\right)e^{2s\Phi_1}dxdt\le C\int_0^T\int_{-\beta_1}^{\beta_1}  H^{2}e^{2s\Phi_1}dxdt.
\]
Now, as in \cite{fm1}, we can prove that exists a positive constant $k$ such that
\begin{equation}\label{prima'}
\tilde a(x) e^{2s\varphi(t,x)} \le k e^{2s\Phi_1(t,x)}
\end{equation}
and
\begin{equation}\label{seconda'}
\frac{(x-x_0)^2}{\tilde a(x)}e^{2s\varphi(t,x)} \le k e^{2s \Phi_1(t,x)}
\end{equation}
for every $(t,x) \in [0, T] \times \left[-\beta_1, {\beta_1}\right]$. 

Thus, by definitions of $\xi$, $W$ and $Z$, we have
\[
\begin{aligned}
&\int_0^T\int_0^{\tilde \lambda_1} \left(s\Theta a(v_x)^2 + s^3 \Theta^3
\frac{(x-x_0)^2}{a} v^2\right)e^{2s\varphi}dxdt
\\& \le \int_0^T\int_{-\beta_1}^{\beta_1} \left(s\Theta \tilde a(Z_x)^2 + s^3 \Theta^3
\frac{(x-x_0)^2}{\tilde a} Z^2\right)e^{2s \varphi}dxdt
\\
&\le C\int_0^T\int_{-\beta_1}^{\beta_1} \left(s\Theta (Z_x)^2 + s^3 \Theta^3
Z^2\right)e^{2s\Phi_1}dxdt\le C\int_0^T\int_{-\beta_1}^{\beta_1}  H^{2}e^{2s\Phi_1}dxdt.
\end{aligned}
\]
Using again the fact that $\xi_x $ is supported in $[-\bar\beta_1,-\tilde\lambda_1]\cup[\tilde\lambda_1, \bar \beta_1]\cup[\bar \lambda_2, \tilde \lambda_2]\cup[2-\tilde \lambda_2, 2 -\bar \lambda_2]$  and the boundedness of $\tilde a'$ (far away from $x_0$  in the weakly degenerate case,  see \eqref{rieccola}, and since $a\in $ $W^{1,\infty}(-1,2)$ in the strongly degenerate one), it follows, by the Caccioppoli's inequality for the nondegenerate case (see, e.g., \cite[Remark 7]{fmJDE}).
\[
\begin{aligned}
&\int_0^T\int_{-\beta_1}^{\beta_1}  H^{2}e^{2s\Phi_1}dxdt = \int_0^T\int_{-\beta_1}^{\beta_1} (\xi \tilde h+ ( \tilde a \xi_x W)_x + \tilde
a\xi_x W_x)^2e^{2s\Phi_1}dxdt\\
&\le C\left( \int_0^T\int_{-\beta_1}^{\beta_1}  \tilde h^2e^{2s\Phi_1}dxdt + \int_0^T\left(\int_{-\bar\beta_1}^{-\tilde\lambda_1}+ \int_{\tilde\lambda_1}^{\bar\beta_1}\right)(W^2 + W_x^2)e^{2s\Phi_1}dxdt \right)\\
&\le C\left( \int_0^T\int_{-\beta_1}^{\beta_1}  \tilde h^2e^{2s\Phi_1}dxdt + \int_0^T\left(\int_{-\beta_1}^{-\lambda_1}+ \int_{\lambda_1}^{\beta_1}\right)W^2 dxdt \right)\\
&\le C \left(\int_0^T \int_{-\beta_1}^{\beta_1}\tilde h^2 e^{2s \Phi_1}dxdt +  \int_0^T\int_{\lambda_1}^{\beta_1}W^2dxdt\right)\\
&\le  C \int_0^T \int_{-\beta_1}^{\beta_1}
\tilde h^2 e^{2s \Phi_1}dxdt + C \int_0^T\int_\omega v^2dxdt.
\end{aligned}
\]
Now, observe that
\begin{equation}\label{hnew}
\int_0^T\int_{-\beta_1}^{\beta_1} \tilde h^2(t,x)e^{2s \Phi_1(t,x)} dxdt \le \int_0^T\int_0^{\beta_1}h^2 (t,x)e^{2s \Phi_1(t,-x)} dxdt.
\end{equation}
Indeed, using a change of variable and the definition of $\tilde h$, it results
\[
\begin{aligned}
&\int_{-\beta_1}^{\beta_1} \tilde h^2(t,x)e^{2s\Phi_1(t,x)} dx = \int _{-\beta_1}^0 \tilde h^2(t,x)e^{2s \Phi_1(t,x)} dx+ \int_0^{\beta_1}\tilde h^2(t,x)e^{2s\Phi_1(t,x)} dx\\
&=\int _0^{\beta_1} \tilde h^2(t,-y)e^{2s \Phi_1(t,-y)} dy+ \int_0^{\beta_1} h^2(t,x)e^{2s \Phi_1(t,x)} dx\\
&\le\int _0^{\beta_1}  h^2(t,y)e^{2s \Phi_1(t,-y)} dy+ \int_0^{\beta_1} h^2(t,x)e^{2s \Phi_1(t,-x)} dx,
\end{aligned}
\]
since one has $\rho_{-\beta_1, \beta_1} (x) \le \rho_{-\beta_1, \beta_1}(-x),$  for all $x \in [0,\beta_1]$.
Hence, using the definitions of $\Phi_1$, $\tilde a$, $\tilde h$ and $W$, it results
\begin{equation}\label{sopra0}
\begin{aligned}
&\int_0^T\int_0^{\tilde \lambda_1} \left(s\Theta a(v_x)^2 + s^3 \Theta^3
\frac{(x-x_0)^2}{a} v^2\right)e^{2s\varphi}dxdt \\&\le C\left( \int_0^T \int_0^{\beta_1}
 h^2 e^{2s \Phi_1(t,-x)}dxdt+\int_0^T \int_\omega v^2dxdt\right)
\end{aligned}
\end{equation}
for all $s \ge s_0$. Analogously, we can choose $s_0$ so large that, for all $s \ge s_0$ and for a positive constant $C$:
\[
\begin{aligned}
&\int_0^T\int_{\tilde\lambda_2}^1 \left(s\Theta a(v_x)^2 + s^3 \Theta^3
\frac{(x-x_0)^2}{a} v^2\right)e^{2s\varphi}dxdt\\
&\le  C \left(\int_0^T \int_{\lambda_2}^1 h^2e^{2s \Phi_2(t,x)}dxdt + \int_0^T \int_\omega v^2dxdt\right),
\end{aligned}
\]
since $\rho_{\lambda_2, 2-\lambda_2}(2-x) \le  \rho_{\lambda_2, 2-\lambda_2}(x)$ for all $x \in [\lambda_2, 1]$.
Finally, consider a
smooth function $\eta: [0,1] \to [0,1]$ such that 
\begin{equation}\label{eta}
\eta(x):= \begin{cases}0, & x \in [0, \bar\lambda_1]\cup [\bar \beta_2,1],\\
1, & x \in [\tilde \lambda_1, \tilde\lambda_2].
\end{cases}
\end{equation}
Define $w:= \eta v$. Then $w$ satisfies
\begin{equation}\label{p1''}
\begin{cases}
w_t + \left(a w_x \right) _x =\bar h, & (t,x) \in (0,T) \times (0,1),\\
w(t,0)= w(t, 1)=0, &  t \in (0,T),\\
\end{cases}
\end{equation}
 with $\bar h:=\eta h+ (a \eta_x v)_x + 
a\eta_x v_x$
Hence, by Theorem \ref{Cor1fm1} and \cite[Proposition 5.2]{fm1}, we have
\[
\begin{aligned}
&\int_0^T\int_{\tilde\lambda_1}^{\tilde\lambda_!} \left(s\Theta a(v_x)^2 + s^3 \Theta^3
\frac{(x-x_0)^2}{a} v^2\right)e^{2s\varphi}dxdt
\\
&\le  C \int_0^T \int_0^1 \left(s\Theta a(w_x)^2 + s^3 \Theta^3
\frac{(x-x_0)^2}{a} w^2\right)e^{2s\varphi}dxdt\\
 & \le  C \int_0^T \int_0^1
 (\eta  h+ ( a \eta_x v)_x + 
a\eta_x v_x)^2e^{2s\varphi}dxdt \\
 & \le C \left(\int_0^T \int_0^1
 h^2 e^{2s\varphi}dxdt + \int_0^T \left(\int_{\bar \lambda_1}^{\tilde\lambda_1}+ \int_{\tilde\lambda_2}^{\bar\beta_2}\right)(v^2+v_x^2)e^{2s  \varphi}dxdt \right)\\
& \le C\left( \int_0^T \int_0^1
 h^2e^{2s\varphi} dxdt + \int_0^T \int_\omega v^2dxdt\right) .
\end{aligned}
\]
Hence,  we can choose $s_0$ so large that, for all $s \ge s_0$ and for a positive constant $C$,  
\[
\begin{aligned}
&\int_0^T\int_0^1 \left(s\Theta a(v_x)^2 + s^3 \Theta^3
\frac{(x-x_0)^2}{a} v^2\right)e^{2s\varphi}dxdt\le C  \int_0^T \int_\omega v^2dxdt\\
&+ C\left(\int_0^T \int_0^1
 h^2e^{2s\varphi}dxdt  + \int_0^T \int_0^{\beta_1}
 h^2 e^{2s \Phi_1(t,-x)}dxdt+ \int_0^T \int_{\lambda_2}^1 h^2e^{2s \Phi_2(t,x)}dxdt\right).
\end{aligned}
\]
Nothing changes in the proof if $\omega= \omega_1\cup \omega_2$ and each of these intervals lye on different sides of $x_0$, as the assumption implies.
\end{proof}

We underline that, in the weakly degenerate case, the boundedness of $a'$ far away from $x_0$ is crucial in the previous proof. Indeed,  thanks to it, we are able to estimate the integral $\int_0^T\left(\int_{-\bar\beta_1}^{-\tilde \lambda_1}+ \int_{\tilde \lambda_1}^{\bar\beta_1}\right)[( \tilde a \xi_x W)_x]^2e^{2s\Phi_1}dxdt$.
\section{Carleman estimate for degenerate parabolic
problems: the non divergence case}\label{Carleman estimate nondiv}
In this section we prove the analogue of the Carleman estimate given in Theorem \ref{Cor1} for the
adjoint problem of \eqref{linear} in the non divergence case, when the degeneracy is weak or strong:
\begin{equation}\label{1'}
\begin{cases}
v_t + av_{xx} =h, & (t,x) \in Q_T,\\
v_x(t,0)=v_x(t,1)=0, &  t \in (0,T).\\
\end{cases}
\end{equation}
Here $h \in L^2_{\frac{1}{a}}(Q_T),$
while on $a$ we make the following assumptions:
\begin{Assumptions}\label{Ass04}
The function $a$ is such that
\begin{enumerate}
\item the operator $A_2$ is weakly or strongly degenerate;
\item the function $\displaystyle \frac{(x-x_0)a'(x)}{a(x)} \in W^{1,\infty}(0,1)$;
\item  in the weakly degenerate case there exist $B_1\in (0,x_0)$, $B_2\in (1,2-x_0)$,
two functions $\fg \in L^\infty_{\rm loc}((-x_0, 2-x_0)\setminus \{x_0\})$, $\fh (\cdot, B_i) \in W^{1,\infty}_{\rm loc}((-x_0, 2-x_0)\setminus \{x_0\})$ and
two strictly positive constants $\fg_0$, $\fh_0$ such that $\fg(x) \ge \fg_0$ and
\begin{equation}\label{5.3'}
\frac{\tilde a'(x)}{2\sqrt{\tilde a(x)}}\left(\int_x^{B_i}\fg(t) dt + \fh_0 \right)+ \sqrt{\tilde a(x)}\fg(x) =\fh(x,B_i)\quad \text{for a.e.} \; x \in (-x_0, 2-x_0)
\end{equation}
with $i=1,2$, $-x_0< x<B_1$ or $x_0 <x <B_2$, and
 $\tilde a$  as in \eqref{tildea};
\item if $K \ge \displaystyle
\frac{1}{2}$ \eqref{dainfinito} holds.
 \end{enumerate}
\end{Assumptions}

\begin{Remark}
We underline the fact that in the non
divergence case the assumptions on  $a$ are weaker than in the divergence case: the additional condition when $K > 3/2$
is not necessary, since all integrals and integrations by parts are
justified by definition of $D(A_2)$.

Moreover, Hypothesis \ref{Ass03}$.3$ is substituted by Hypothesis \ref{Ass04}$.4$, which is essential to prove \cite[Theorem 4.2]{fm1} (see \cite[Lemma 4.3]{fm1} and \cite[Lemma 3.10]{cfr} or \cite[Lemma 5]{cfr1} for the case when the degeneracy occurs at the boundary of the domain).
\end{Remark}

\begin{Remark}
As in Remark \ref{remspostata} we can take $\fg_0=\min \left\{1, \frac{-\alpha(B_1+1)+x_0}{x_0}, \frac{-\alpha B_2+ 1-x_0}{1-x_0}\right\}$, $\fh_0=1 $, $
\fg(x) = 1$ in $[0,x_0]\cup[1,2-x_0)$, while 
$
\fg(x) =
\frac{-\alpha(B_1+1)+x_0}{x_0}=:L_1$ in $(-x_0, 0)$ and 
$
\fg(x) =\frac{-\alpha B_2+ 1-x_0}{1-x_0}=:L_2$ in $(x_0, 1)$.
On the other hand  
\[\fh (x,B_1)= \begin{cases}\ds (x_0-x)^{\frac{\alpha}{2}-1}\left[- \frac{\alpha}{2} (B_1+1-x) + (x_0-x)\right], & x \in [0, B_1),\\
(x+x_0)^{\frac{\alpha}{2}-1}\left[\ds\frac{\alpha}{2} (-L_1x+B_1+1) + L_1(x+x_0)\right], & x  \in (-x_0, 0)
\end{cases}
\]
and
\[\fh (x,B_2)= \begin{cases}
(x-x_0)^{\frac{\alpha}{2}-1} \left[\ds\frac{\alpha}{2}(L_2(1-x) + B_2) + L_2(x-x_0)\right], & x  \in (x_0, 1)\\
\ds (2-x-x_0)^{\frac{\alpha}{2}-1}\left[- \frac{\alpha}{2} (B_2+1-x) + (2-x-x_0)\right], & x \in [1, B_2),\\
\end{cases}
\]
 Again $\fg \in L^\infty_{\rm loc}((-x_0, 2-x_0)\setminus \{x_0\})$ and $\fh \in W^{1,\infty}_{\rm loc}((-x_0, 2-x_0)\setminus \{x_0\}; L^\infty(0,1))$.
\end{Remark}
To prove an estimate of Carleman type, we proceed as before. To this aim,
as in \cite[Chapter 4]{fm1}, let us introduce the function
\begin{equation}\label{gammabo}
\gamma(t,x):
=\Theta(t)\mu(x),
\end{equation}
where $\Theta$ is as in \eqref{c_1} and
\begin{equation}\label{c_1'}
\mu(x) := d_1\left(\int_{x_0}^x \frac{y-x_0}{a(y)}e^{R(y-x_0)^2}dy-
d_2\right).
\end{equation}
Here $d_2> \displaystyle
\max\left\{\frac{(1-x_0)^2e^{R(1-x_0)^2}}{(2-K)a(1)},
\frac{x_0^2e^{Rx_0^2}}{(2-K)a(0)}\right\}$, $R$ and $d_1$ are strictly positive constants.
The main result of this section is the following:
\begin{Theorem}\label{Cor2}
Assume Hypothesis  $\ref{Ass04}$. Let $\omega\subset \subset (0,1)$ be an open interval containing $x_0$, $B_1$ and $2-B_2$, or let $\omega= \omega_1\cup \omega_2$, where $\omega_i=(\lambda_i,\beta_i) \subset (0,1), \, i=1,2,$
$\beta_1\leq B_1$ and $2-B_2\le \lambda_2$. Then, there
exist two positive constants $C$ and $s_0$, such that every solution
$v$ of \eqref{1'} in
$
\mathcal{S}:=
L^2\big(0, T; H^2_{\frac{1}{a}}(0,1)\big)\cap H^1\big(0, T;H^1_{\frac{1}{a}}(0,1)\big) 
$
satisfies
\begin{equation}\label{car'}
\begin{aligned}&\int_0^T\int_0^1 \left(s\Theta (v_x)^2 + s^3 \Theta^3
\left(\frac{x-x_0}{a} \right)^2v^2\right)e^{2s\gamma}dxdt
 \le C  \int_0^T \int_\omega \frac{v^2}{a}dxdt\\
&+C\left( \int_0^T\int_0^1 \frac{h^{2}}{a}e^{2s\gamma(t,x)}dxdt
+  \int_0^T \int_0^{\beta_1}
 h^2 e^{2s \Phi_1(t,-x)}dxdt+ \int_0^T \int_{\lambda_2}^1 h^2e^{2s \Phi_2(t,x)}dxdt\right)
\end{aligned}
\end{equation}
for all $s \ge s_0$.
\end{Theorem}

To prove Theorem \ref{Cor2}, we will use the Carleman estimate given in  \cite[Theorem 4.2]{fm1} for the analogous problem of \eqref{1'} with Dirichlet boundary conditions:
\begin{Theorem}\label{Cor11fm1}
Assume Hypothesis  $\ref{Ass04}$. Then,
there exist two positive constants $C$ and $s_0$ such that every
solution $v\in 
L^2\big(0, T; \mathcal H^2_{\frac{1}{a}}(0,1)\big)\cap H^1\big(0, T;\mathcal H^1_{\frac{1}{a}}(0,1)\big) $ of
\begin{equation}\label{vprova}
\begin{cases}
v_t + av_{xx} =h & (t,x) \in Q_T,\\
v(t,0)=v(t,1)=0 &  t \in (0,T),\\
\end{cases}
\end{equation}
satisfies, for all $s \ge s_0$,
\begin{equation}\label{car}
\begin{aligned}
&\int_{Q_T} \left(s\Theta (v_x)^2 + s^3 \Theta^3
\left(\frac{x-x_0}{a} \right)^2v^2\right)e^{2s\varphi}dxdt\\
&\le C\left(\int_{Q_T} h^{2}\frac{e^{2s\varphi}}{a}dxdt +
sd_1\int_0^T\left[\Theta e^{2s \varphi}(x-x_0)(v_x)^2
dt\right]_{x=0}^{x=1}\right),
\end{aligned}
\end{equation}
where $d_{1}$ is the constant introduced in
\eqref{c_1'}. Here
\[
\mathcal H^1_{\frac{1}{a}}(0,1) :=L^2_{\frac{1}{a}}(0,1)\cap H^1_0(0,1),
\]
and
\[
\mathcal H^2_{\frac{1}{a}}(0,1) :=\Big\{ u \in \mathcal H^1_{\frac{1}{a}}(0,1) \,
\big| \, u'\in H^1(0,1)\Big\}.
\]
\end{Theorem}
\begin{proof}[Proof of Theorem \ref{Cor2}]
The proof is similar to the one of Theorem \ref{Cor1}. So we sketch it. To this aim consider $\bar \lambda_i, \tilde \lambda_i, \bar \beta_i$ ($i=1,2$), $\xi, \eta$, $W$ and $Z$
as in the proof of Theorem \ref{Cor1}. Obviously, $W$ and $Z$ satisfy, respectively, the following problems
\[
\begin{cases}
W_t + \tilde a W_{xx} =\tilde h, & (t,x) \in (0,T) \times (-1,2),\\
W_x(t,-1)=W_x(t,2)=0, &  t \in (0,T),
\end{cases}
\]
\begin{equation}\label{p2}
\begin{cases}
Z_t + \tilde a Z_{xx} =H, & (t,x) \in (0,T) \times (-\beta_1, \beta_1),\\
Z(t,-\beta_1)=Z(t,\beta_1), & t \in (0,T),\\
\end{cases}
\end{equation}
and
\[
\begin{cases}
Z_t + \tilde a Z_{xx} =H, & (t,x) \in (0,T) \times (\lambda_2, 2-\lambda_2),\\
Z(t,\lambda_2)= Z(t, 2-\lambda_2)=0, &  t \in (0,T),\\
\end{cases}
\]
 with $H:=\xi \tilde h+ \tilde a (\xi _{xx} W + 2\xi _x  W_x)$, being $\tilde a$ and $\tilde h$ defined as before. Observe that $Z_x(t, -\beta_1)= Z_x(t, \beta_1)=Z_x(t, \lambda_2)=Z_x(t,2-\lambda_2)=0$
and, by  the assumption on $a$, $H\in L^2((0,T); L^2_{\frac{1}{\tilde a}} (I)),$ where $I$ is as before.
Thus, we can apply the
analogue of \cite[Theorem 3.2]{fm1} on $(-\beta_1,\beta_1)$ in place of
$(0,1)$ and with weight $\Phi_1$, obtaining that there exist
two positive constants $C$ and $s_0$ ($s_0$ sufficiently large), such that, for all $s \ge s_0$,
\[
\int_0^T\int_{-\beta_1}^{\beta_1} \left(s\Theta (Z_x)^2 + s^3 \Theta^3
Z^2\right)e^{2s\Phi_1}dxdt\le C\int_0^T\int_{-\beta_1}^{\beta_1}  H^{2}e^{2s\Phi_1}dxdt.
\]
By definitions of $\xi$, $W$, $Z$ and by \eqref{hnew}, proceeding as before, we have
\[
\begin{aligned}
&\int_0^T\int_0^{\tilde \lambda_1} \left(s\Theta (v_x)^2 + s^3 \Theta^3
\left(\frac{x-x_0}{a} \right)^2v^2\right)e^{2s\gamma}dxdt
\\&\le \int_0^T\int_{-\beta_1}^{\beta_1} \left(s\Theta (Z_x)^2 + s^3 \Theta^3
\left(\frac{x-x_0}{\tilde a} \right)^2 Z^2\right)e^{2s \gamma}dxdt
\\
&\le C \int_0^T\int_{-\beta_1}^{\beta_1} \left(s\Theta (Z_x)^2 + s^3 \Theta^3
Z^2\right)e^{2s\Phi_1}dxdt\le C\int_0^T\int_{-\beta_1}^{\beta_1}  H^{2}e^{2s\Phi_1}dxdt\\
&\le C\left( \int_0^T \int_{-\beta_1}^{\beta_1}
\tilde h^2 e^{2s \Phi_1}dxdt + \int_0^T\int_{\lambda_1}^{\beta_1}W^2dxdt\right)\\
&\le C\left(\int_0^T \int_0^{\beta_1}
 h^2 e^{2s\Phi_1(t,-x)} dxdt + \int_0^T\int_{\lambda_1}^{\beta_1}\frac{v^2}{a}dxdt\right).
\end{aligned}
\]
for all $s \ge s_0$. Analogously, we can chose $s_0$ so large that, for all $s \ge s_0$ and for a positive constant $C$:
\[
\begin{aligned}
&\int_0^T\int_{\tilde \lambda_2}^1 \left(s\Theta (v_x)^2 + s^3 \Theta^3
\left(\frac{x-x_0}{a} \right)^2 v^2\right)e^{2s\gamma}dxdt
\\
&\le  C \left(\int_0^T \int_{\lambda_2}^1 h^2e^{2s \Phi_2(t,x)}dxdt + \int_0^T \int_{\lambda_2}^{\beta_2}\frac{v^2}{a}dxdt\right).
\end{aligned}
\]
Finally, consider $w:= \eta v$. Then $w$ satisfies
\[
\begin{cases}
w_t + a w_{xx} =\bar h, & (t,x) \in (0,T) \times (0,1),\\
w(t,0)= w(t, 1)=0, &  t \in (0,T),\\
\end{cases}
\]
 with $\bar h:=\eta h+  a (\eta _{xx} v + 2\eta_x  v_x)$.
Hence, by Theorem \ref{Cor11fm1} and \cite[Proposition 5.4]{fm1}, we have
\[
\begin{aligned}
&\int_0^T\int_{\tilde \lambda_1}^{\tilde\lambda_2} \left(s\Theta (v_x)^2 + s^3 \Theta^3
\left(\frac{x-x_0}{a} \right)^2 v^2\right)e^{2s\gamma}dxdt
\\
&\le  C \int_0^T \int_0^1 \left(s\Theta (w_x)^2 + s^3 \Theta^3
\left(\frac{x-x_0}{a} \right)^2w^2\right)e^{2s\gamma}dxdt \\
&\le C\left( \int_0^T \int_0^1
\frac{h^2 }{a}e^{2s\gamma} dxdt + + \int_0^T \left(\int_{\bar\lambda_1} ^{\tilde \lambda_1}+\int_{\tilde\lambda_2}^{\bar \beta_2}\right)v^2dxdt\right)\\
&\le C\left( \int_0^T \int_0^1
\frac{h^2 }{a}e^{2s\gamma} dxdt + + \int_0^T \left(\int_{\bar\lambda_1} ^{\tilde \lambda_1}+\int_{\tilde\lambda_2}^{\bar \beta_2}\right)\frac{v^2}{a}dxdt\right).
\end{aligned}
\]
Hence, we can choose $s_0$ so large that, for all $s \ge s_0$ and for a positive constant $C$, 
\[
\begin{aligned}
&\int_0^T\int_0^1 \left(s\Theta (v_x)^2 + s^3 \Theta^3
\left(\frac{x-x_0}{a} \right)^2 v^2\right)e^{2s\gamma}dxdt \le C  \int_0^T \int_\omega \frac{v^2}{a}dxdt
\\
&+ C\left( \int_0^T \int_0^1
\frac{h^2}{a}e^{2s\gamma(t,x)} + \int_0^T \int_0^{\beta_1}
 h^2 e^{2s \Phi_1(t,-x)}dxdt+ \int_0^T \int_{\lambda_2}^1 h^2e^{2s \Phi_2(t,x)}dxdt\right).
\end{aligned}
\]\end{proof}
\section{Observability inequalities as applications of Carleman estimates}\label{secobserv}

In this section we consider problem \eqref{linear1} and we make the following assumptions:

we assume that the control set $\omega$ is an interval which contains the degeneracy point
or the union of  two intervals each of them
lying on one side of the degeneracy point, i.e.
\begin{itemize}
\item 
\begin{equation}\label{omega1}
\omega=(\alpha,\beta) \subset (0,1) \mbox{ is such that $x_0 \in
\omega$},
\end{equation}
or
\item $\omega = \omega_1 \cup
\omega_2,$ where
\begin{equation}\label{omega2}
\omega_i=(\lambda_i,\beta_i) \subset (0,1), \, i=1,2, \mbox{ and
$\beta_1 < x_0< \lambda_2$}.
\end{equation}
\end{itemize}
\begin{Remark}\label{beta1}
Observe that, if \eqref{omega1} holds, we can find two subintervals
$\omega_1=(\lambda_1,\beta_1)\subset (\alpha, x_0),
\omega_2=(\lambda_2,\beta_2) \subset (x_0,\beta)$ such that $(\omega_1
\cup \omega_2) \subset \subset \omega \setminus \{x_0\}$. 
\end{Remark}

\begin{Assumptions}\label{hyp6.1} 
Assume Hypothesis $\ref{Ass03}$ (in the divergence case) or Hypothesis $\ref{Ass04}$ (in the nondivergence one) with $B_1= \beta_1$ and $B_2= 2-\lambda_2$.
\end{Assumptions}
\vspace{0.5cm}
Now, we associate to \eqref{linear1} the
homogeneous adjoint problem
\begin{equation}\label{h=0}
\begin{cases}
v_t +Av= 0, &(t,x) \in  Q_T,
\\[5pt]
v_x(t,0)=v_x(t,1) =0, & t \in (0,T),
\\[5pt]
v(T,x)= v_T(x)\in X,
\end{cases}
\end{equation}
where $T>0$ is given and, we recall, $X$ denotes the Hilbert space $L^2(0,1)$ or $L^2_{\frac{1}{a}}(0,1)$ in the divergence or in the non divergence case, respectively.
By the Carleman estimates given in Theorems
\ref{Cor1} and \ref{Cor2}, we will deduce the following observability inequalities
for both the weakly and the strongly degenerate cases:
\begin{Proposition}\label{obser.}
Assume Hypotheses $\ref{hyp6.1}$.
Then there exists a positive constant $C_T$ such that every solution
$v \in  C([0, T]; L^2(0,1)) \cap L^2 (0,T; H^1_a(0,1))$ of
\eqref{h=0} with $Au=A_1u$ satisfies
 \begin{equation}\label{obser1.}
\int_0^1v^2(0,x) dx \le C_T\int_0^T \int_{\omega}v^2(t,x)dxdt.
\end{equation}
\end{Proposition}
\begin{Proposition}\label{obser.1}
Assume Hypotheses $\ref{hyp6.1}$.
Then there exists a positive constant $C_T$ such that every solution
$v \in  C([0, T]; L^2_{\frac{1}{a}}(0,1)) \cap L^2 (0,T; H^1_{\frac{1}{a}}(0,1))$ of
\eqref{h=0} with $Au=A_2u$ satisfies
 \begin{equation}\label{obser1.1}
\int_0^1v^2(0,x) \frac{1}{a}dx \le C_T\int_0^T \int_{\omega}v^2(t,x)\frac{1}{a}dxdt.
\end{equation}
\end{Proposition}
\subsection{Proof of Proposition \ref{obser.}} In
this subsection we will prove, as a consequence of the Carleman
estimate given in Section \ref{Carleman estimate}, the
observability inequality \eqref{obser1.}.
The proof is similar to the one given in \cite{fm} or in \cite[Proposition 5.1]{fm1},  so we sketch it.
Thus, we consider the
adjoint problem with more regular final--time datum
\begin{equation}\label{h=01}
\begin{cases}
v_t +A_1v= 0, &(t,x) \in  Q_T,
\\[5pt]
v_x(t,0)=v_x(t,1) =0, & t \in (0,T),
\\[5pt]
v(T,x)= v_T(x) \,\in D(A_1^2),
\end{cases}
\end{equation}
where $
D(A_1 ^2) = \Big\{u \,\in \,D(A_1 )\;\big|\; A_1 u \,\in
\,D(A_1 ) \;\Big\}.$
Observe that $D(A_1 ^2)$ is densely
defined in $D(A_1 )$ (see, for example, \cite[Lemma 7.2]{b}) and
hence in $L^2(0,1)$. As in \cite{cfr}, \cite{cfr1}, \cite{f}, \cite{fm} or \cite{fm1},
letting $v_T$ vary in $D(A_1 ^2)$, we define the following class
of functions:
\[
\cal{W}_1:=\Big\{ v\text{ is a solution of \eqref{h=01}}\Big\}.\]
Obviously (see, for example, \cite[Theorem 7.5]{b})
$\cal{W}_1\subset
C^1\big([0,T]\:;\:H^2_a(0,1)\big) \subset \mathcal{V} \subset
\cal{U}_1,$
where
$
\cal{U}_1:= C([0,T]; L^2(0,1)) \cap L^2(0, T; H^1_a(0,1)).
$
We shall also need the following lemma, that deals with the
different situations in which $x_0$ is inside or outside the control
region $\omega$. The statements of the conclusions are the
same, however, the
proofs, though inspired by the same ideas, are different. For this reason we divide the proof into two parts.
\begin{Lemma}\label{lemma3}
Assume Hypotheses $\ref{hyp6.1}$. Then there exist two positive
constants $C$ and $s_0$ such that every solution $v \in \cal W_1$ of
\eqref{h=01} satisfies, for all $s \ge s_0$,
\[
\int_0^T\int_0^1\left( s \Theta a (v_{x})^{2} + s^3 \Theta ^3
\frac{(x-x_0)^2}{a} v^{2}\right) e^{{2s\varphi}}  dxdt\le C
\int_0^T\int_{\omega}v^{2} dxdt.
\]
Here $\Theta$ and $\varphi$ are as in Section \ref{Carleman estimate}.
\end{Lemma}

\begin{proof} The proof of Lemma \ref{lemma3} is divided into two parts to distinguish the cases when $\omega$ is an interval which contains the degeneracy point or it is an interval lying on one
side of the degeneracy point.

{\it First case: $\omega=(\alpha,\beta) \subset (0,1)$ is such that $x_0 \in
\omega$.
}

As in the proof of Theorem \ref{Cor1},  fix $\bar \lambda_i$, $\tilde\lambda_i$, $\bar \beta_i$ $(i=1,2)$, and smooth functions $\xi, \eta$ as in \eqref{xi1} and \eqref{eta}.
    Define $w:= \eta v$, where $v$ solves \eqref{h=01}.
   Hence, $w$   satisfies
    \begin{equation}
    \label{eq-w*}
    \begin{cases}
      w_t + (a  w_x) _x =( a \eta _x v )_x + \eta _x a v_x =:f,&
      (t,x) \in(0, T)\times (0,1), \\
   w(t,0)= w(t,1)=0, & t \in (0,T).
    \end{cases}
\end{equation}
Applying Theorem \ref{Cor1}, we have that there exist two positive constants $C$ and $s_0$ such that
\begin{equation}\label{car9}
\begin{aligned}
& \int_0^T \int_0^1\Big( s \Theta  a (w_x)^2 + s^3 \Theta^3
   \frac{(x-x_0)^2}{a} w^2 \Big)
    e^{2s \varphi} \, dx dt  \le C \left(\int_0^T \int_\omega w^2  dxdt\right)
    \\
    &+ C\left(\int_0^T \int_0^1f^2e^{2s\varphi}+ \int_0^T \int_0^{\beta_1}
 f^2 e^{2s \Phi_1(t,-x)}dxdt+ \int_0^T \int_{\lambda_2}^1 f^2e^{2s \Phi_2(t,x)}dxdt\right),
\end{aligned}
\end{equation}
for all $s \ge s_0$. Then, using the definition of $\eta$  and in
particular the fact that  in $[0,1]$  the functions $\eta_x$ and  $\eta_{xx}$ are supported in
$\tilde \omega:= [\bar\lambda_1, \tilde\lambda_1]
\cup[ \tilde \lambda_2,\bar \beta_2]$, we can write
$
f^2=  (( a \xi _x v )_x + \xi _x a v_x)^2 \le C( v^2+
(v_x)^2)\chi_{\tilde \omega},$
since the function $a'$  is bounded on $\tilde \omega$. Hence,
applying the Caccioppoli inequality \cite[Proposition 4.2]{fm} and \eqref{car9}, we get
 \begin{equation}\label{stimacar}
\begin{aligned}
&\int_0^T\int_{\tilde\lambda_1}^{\tilde\lambda_2}\left( s \Theta a (v_x)^{2} + s^3
\Theta ^3 \frac{(x-x_0)^2}{a} v^{2}\right) e^{{2s\varphi}} dxdt
\\&\le \int_0^T \int_0^1\Big( s \Theta  a (w_x)^2 + s^3 \Theta^3
   \frac{(x-x_0)^2}{a} w^2 \Big)
    e^{2s \varphi} \, dx dt
     \\ & \le C\!\int_0^T \!\!\int_\omega v^2dxdt +C \! \int_0^T \!\!\int_{\tilde \omega}e^{2s\varphi}(
v^2+ (v_x)^2)dxdt+C \! \int_0^T \!\!\int_{\bar\lambda_1}^{\tilde\lambda_1}e^{2s\Phi_1(t,-x)}(
v^2+ (v_x)^2)dxdt\\
&+ C  \int_0^T \int_{\tilde\lambda_2}^{\bar\beta_2}e^{2s\Phi_2(t,x)}(
v^2+ (v_x)^2)dxdt \le C \int_0^T \int_{\omega} v^2dxdt,
\end{aligned} \end{equation} for a positive constant $C$.
Now, define   $Z:= \xi W$, where $W$ is defined in \eqref{W}. Then $Z$ is the solution of \eqref{p1} and \eqref{p1'} with $H:=( \tilde a \xi_x W)_x + \tilde a\eta_x W_x$.
   Proceeding as in proof of Theorem \ref{Cor1} and using the fact that in $[0,2]$ the functions $\xi_x, \xi_{xx}$ are supported in $\ds \left[\tilde \lambda_1, \bar \beta_1\right]\cup \left[\bar \lambda_2, \tilde \lambda_2\right]\cup \left[2-\tilde \lambda_2, 2 -\bar \lambda_2\right]$
we get
\begin{equation}\label{stimacar2}
\begin{aligned}
&\int_0^T\int_{\tilde \lambda_2}^1 \left(s\Theta a(v_x)^2 + s^3 \Theta^3
\frac{(x-x_0)^2}{a} v^2\right)e^{2s\varphi}dxdt\\
&\le \int_0^T\int_{\lambda_2}^{2-\lambda_2} \left(s\Theta \tilde a(Z_x)^2 + s^3 \Theta^3
\frac{(x-x_0)^2}{\tilde a} Z^2\right)e^{2s\varphi}dxdt
\\
&\le \int_0^T\int_{\lambda_2}^{2-\lambda_2} \left(s\Theta (Z_x)^2 + s^3 \Theta^3
 Z^2\right)e^{2s\Phi_2}dxdt
\le  C \int_0^T \int_\omega v^2dxdt.
 \end{aligned}
\end{equation}
Thus \eqref{stimacar} and \eqref{stimacar2} imply
    \begin{equation}\label{carin0}
    \begin{aligned}
      \int_{0}^T \int _{\tilde \lambda_1}^{1}  \Big( s \Theta  a (v_x)^2 + s^3 \Theta^3 \frac{(x-x_0)^2}{a}  v^2 \Big)
      e^{2s \varphi } \, dx dt
     \le  C \int_{0}^T \int _{\omega}   v^2  dxdt,
    \end{aligned}
\end{equation}
for some positive constant $C$. To complete the proof it is
sufficient to prove a similar inequality on the interval
$[0,\tilde\lambda_1]$. To this aim, we follow a reflection procedure as before
considering the problem \eqref{p1}. Hence, using the Caccioppoli's inequality (see, e.g., \cite[Proposition 5.2.]{fm1} and \cite[Theorem 3.1]{fm1},
\begin{equation}\label{stimacar20}
\begin{aligned}
& \int_0^T\int_{-\beta_1}^{\beta_1} \Big(s \Theta \tilde a (Z_x)^2+
s^3 \Theta^3  \frac{(x-x_0)^2}{\tilde a}Z^2\Big)
e^{2s\varphi}dxdt\\&
     \le k\int_0^T\int_{-\beta_1}^{\beta_1}( s \Theta (Z_x)^2e^{2s\Phi_1}+
s^3\Theta^3Z^2e^{2s\Phi_1} )dxdt  \le C
\int_0^T\int_{-\beta_1}^{\beta_1}e^{2s\Phi_1} H^2 dxdt \\
&\le C \int_0^T
\int_{-\bar\beta_1}^{-\tilde\lambda_1}e^{2s\Phi_1}( W^2+
(W_x)^2)dxdt + C\int_0^T \int_{\tilde\lambda_1}^{\bar\beta_1}e^{2s\Phi_1}(W^2+ (W_x)^2)dxdt\\
&\le C \int_0^T \int_{\lambda_1}^{\beta_1}W^2dxdt
\leq C \int_0^T \int_{\lambda_1}^{\beta_1}
v^2dxdt \le C \int_0^T \int_{\omega}
v^2dxdt,
\end{aligned}
\end{equation}
for all $s\geq s_0$.
Hence, by \eqref{stimacar20} and the definitions of $W$ and $Z$, we
get
\begin{equation}\label{car10}
\begin{aligned}
&\int_0^T\int_0^{\tilde\lambda_1}  \Big( s^3 \Theta^3
      \frac{(x-x_0)^2}{a}v^2+s \Theta a (v_x)^2\Big) e^{2s\varphi}dxdt\\
&
\le \int_0^T\int_{-\beta_1}^{\beta_1} \Big( s^3 \Theta^3
      \frac{(x-x_0)^2}{\tilde a}Z^2+s \Theta \tilde a (Z_x)^2\Big) e^{2s\varphi}dxdt
    \le C \int_0^T \int_{\omega} v^2dxdt,
\end{aligned}
\end{equation}
for a positive constant $C$.
Therefore, by \eqref{carin0} and \eqref{car10}, Lemma \ref{lemma3}
follows.

{\it Second case: $\omega$ is such that \eqref{omega2} holds.
}

Fix $\tilde \lambda_i, \bar \lambda_i, \bar \beta_i  \in (\lambda_i, \beta_i)$, $i=1,2$ and the smooth
functions $\xi$, $\eta$ as before.
Then, define $w:= \eta v$, where $v$ is any fixed solution of
\eqref{h=01}. Hence $w$ satisfies \eqref{eq-w*} and
$
f^2 \le C(
v^2+ (v_x)^2)\chi_{\tilde \omega},
$
where $\tilde \omega$ is as in the first case.
Applying Theorem \ref{Cor1} to $w$, we have that there exist two positive constants $C$ and $s_0$ such that
\begin{equation}\label{car9'}
\begin{aligned}
&\int_0^T \int_0^1\Big( s \Theta a  (w_x)^2 + s^3 \Theta^3
\frac{(x-x_0)^2}{a}\ w^2 \Big) e^{2s \varphi} \, dx dt \le C
\int_0^T \int_{\omega}w^2 dxdt\\
& + C\left(\int_0^T \int_0^1f^2e^{2s\varphi}dxdt +\int_0^T \int_0^{\beta_1}
 f^2 e^{2s \Phi_1(t,-x)}dxdt+ \int_0^T \int_{\lambda_2}^1 f^2e^{2s \Phi_2(t,x)}dxdt\right),
\end{aligned}
\end{equation}
for all $s \ge s_0$. Hence, using \cite[Proposition 4.2]{fm}, we find
\begin{equation}\label{6.14'}
\begin{aligned}
&\int_0^T\int_{\tilde \lambda_1}^{\tilde \lambda_2}\left( s \Theta a(v_x)^{2} + s^3
\Theta ^3 \frac{(x-x_0)^2}{a} v^{2}\right)
e^{{2s\varphi}} dxdt\\
&\le \int_0^T \int_0^1\Big( s \Theta  a(w_x)^2 + s^3
\Theta^3 \frac{(x-x_0)^2}{a}w^2 \Big)
e^{2s \varphi} \, dx dt\\
& \le C\int_0^T\int_\omega v^2 dxdt + C  \int_0^T \int_{\tilde \omega}e^{2s\varphi}( v^2+
(v_x)^2)dxdt \\
&+ C\left(\int_0^T \int_{\bar\lambda_1}^{\tilde\lambda_1}
 (v^2+v_x^2) e^{2s \Phi_1(t,-x)}dxdt+ \int_0^T \int_{\tilde\lambda_2}^{\bar \beta_2} (v^2+v_x^2) e^{2s \Phi_2(t,x)}dxdt\right)\\
&\le C \int_0^T \int_{\omega} v^2 dxdt.
\end{aligned}
\end{equation}
Finally, define
$Z:= \xi W$, where $W$ is given in \eqref{W}. As before,
\begin{equation}\label{carin0'}
\begin{aligned}
\int_0^T \int _{\tilde\lambda_2}^{1}  \Big( s \Theta  a (v_x)^2 + s^3
\Theta^3 \frac{(x-x_0)^2}{a}  v^2 \Big) e^{2s \varphi } \, dx dt \le
C \int_{0}^T \int _{\omega}   v^2  dxdt,
\end{aligned}
\end{equation}
for some positive constant $C$ and $s\geq s_0$.
To complete the proof it is sufficient to prove a similar inequality
for $x\in[0,\tilde\lambda_1]$.
Using a reflection procedure as in the first part of the proof
and applying \cite[Theorem 3.1]{fm1}, one has
\begin{equation}\label{car10'}
\begin{aligned}
\int_0^T\int_{0}^{\tilde\lambda_1}  \left(s\Theta a(v_x)^2 + s^3 \Theta^3
\frac{(x-x_0)^2}{a} v^2\right)e^{2s\varphi}dxdt\le C \int_0^T \int_{\omega} v^2dxdt,
\end{aligned}
\end{equation}
for a positive constant $C$ and $s$ large enough. Therefore, by
\eqref{6.14'}, \eqref{carin0'} and \eqref{car10'}, the conclusion follows.
\end{proof}

We underline that to prove Lemma \ref{lemma3} a crucial role is played by the Carleman estimates stated in \cite[Theorem 3.1]{fm1} for  non degenerate parabolic problems with  {\it non smooth coefficient}. Moreover, in order to apply such a result equation \eqref{rieccola} is essential.

\vspace{0.5cm}
Using Lemma \ref{lemma3}, we obtain the following result which is crucial to prove Proposition \ref{obser.}:
\begin{Lemma}\label{obser.regular}
Assume Hypotheses $\ref{hyp6.1}$. Then there
exists a positive constant $C_T$ such that every solution $v \in
\cal W_1$ of \eqref{h=01} satisfies
\eqref{obser1.}.
\end{Lemma}
\begin{proof}
The proof  is similar to the one of \cite[Lemma 5.3]{fm1}, but we quickly repeat it for the reader's convenience.

Multiplying the equation of \eqref{h=01} by $v_t$ and integrating by
parts over $(0,1)$, one has
\[
\begin{aligned}
&0 = \int_0^1(v_t+ (av_x)_x)v_t dx= \int_0^1 (v_t^2+ (av_x)_xv_t )dx
= \int_0^1v_t^2dx + \left[av_xv_t \right]_{x=0}^{x=1} \\&-
\int_0^1av_xv_{tx} dx= \int_0^1v_t^2dx -
\frac{1}{2}\frac{d}{dt}\int_0^1a(v_x)^2
 \ge - \frac{1}{2}
\frac{d}{dt}\int_0^1 a(v_x)^2dx.
\end{aligned}
\]
Thus, the function $t \mapsto \int_0^1 a(v_x)^2 dx$ is increasing
for all $t \in [0,T]$. In particular,
$$
\int_0^1 av_x(0,x)^2dx \le \int_0^1av_x(t,x)^2dx \mbox{ for every
}t\in[0,T].
$$
Integrating the last inequality over $\displaystyle\left[\frac{T}{4},
\frac{3T}{4} \right]$ and using Lemma \ref{lemma3} we have that there exists a positive constant $C$ such that
\[
\begin{aligned}
\int_0^1a(v_x)^2(0,x) dx &\le
\frac{2}{T}\int_{\frac{T}{4}}^{\frac{3T}{4}}\int_0^1a(v_x)^2(t,x)dxdt\\&\le
C_T \int_{\frac{T}{4}}^{\frac{3T}{4}}\int_0^1s\Theta
a(v_x)^2(t,x)e^{2s\varphi}dxdt
\le C \int_0^T \int_{\omega}v^2dxdt.
\end{aligned}
\]

Applying
the Hardy- Poincar\'{e} inequality given in \cite[Proposition 2.3]{fm} and the previous inequality, one has
\[
\begin{aligned}
\int_0^1 \left(\frac{a}{(x-x_0)^2}\right)^{1/3}v^2(0,x)dx &\leq
\int_0^1 \frac{p}{(x-x_0)^2} v^2(0,x)dx  \\
&\le C_{HP} \int_0^1 p(v_x)^2(0,x) dx \\&\le \max\{C_1, C_2\}C_{HP}
\int_0^1a(v_x)^2(0,x) dx\\& \le C \int_0^T\int_{\omega}v^2dxdt,
\end{aligned}
\]
for a positive constant $C$. Here $p(x) = (a(x)|x-x_0|^4)^{1/3}$ if
$K > \displaystyle\frac{4}{3}$ or $\displaystyle p(x) =\max_{[0,1]}
a|x-x_0|^{4/3}$ otherwise,
$$
C_1:=
\max\left\{\displaystyle\left(\frac{x_0^2}{a(0)}\right)^{2/3},\displaystyle\left(\frac{(1-x_0)^2}{a(1)}\right)^{2/3}\right\},
$$
$C_2:=
\max\left\{\displaystyle\frac{x_0^{4/3}}{a(0)},\displaystyle\frac{(1-x_0)^{4/3}}{a(1)}\right\}$
and $C_{HP}$ is the Hardy-Poincar\'{e} constant.

By \cite[Lemma 2.1]{fm}, the function $\displaystyle x\mapsto
\frac{a(x)}{(x-x_0)^2}$ is nondecreasing on $[0, x_0)$ and
nonincreasing on $(x_0,1]$; then
\[
\left(\frac{a(x)}{(x-x_0)^2}\right)^{1/3}\ge
C_3:=\min\left\{\left(\frac{a(1)}{(1-x_0)^2}\right)^{1/3},
\left(\frac{a(0)}{x_0^2}\right)^{1/3}\right\} >0.
\]
Hence
\[
C_3\int_0^1v(0,x)^2dx \le C \int_0^T\int_{\omega}v^2dxdt\] and the
thesis follows.

\end{proof}
Using Lemma \ref{obser.regular} and proceeding as in \cite[Proposition 5.1]{fm1}, one can prove Proposition \ref{obser.}.

\subsection{Proof of Proposition \ref{obser.1}} As for the proof of Proposition \ref{obser.1}, we consider again the adjoint problem \eqref{h=01} where the operator $A_1$ is replaced by $A_2$. In this case,
\[
\cal{W}_2:=\Big\{  v \text{ is a solution of \eqref{h=01}, with } A_2 \text{ in place of }A_1 \Big\}
\]
with $\cal{W}_2\subset
C^1\big([0,T]\:;\:H^2_{\frac{1}{a}}(0,1)\big) \subset \mathcal{S} \subset
\cal{U}_2,$ and
$
\cal{U}_2:= C([0,T]; L^2_{\frac{1}{a}}(0,1)) \cap L^2(0, T; H^1_{\frac{1}{a}}(0,1)).
$
As in \cite[Lemma 5.4]{fm1}, one can prove
\begin{Lemma}\label{lemma3''}
Assume Hypotheses $\ref{hyp6.1}$. Then there exist
two positive constants $C$ and $s_0$ such that every solution $v \in
\cal W_2$ of \eqref{h=01} satisfies
\[
\int_0^T\int_0^1\left( s \Theta  (v_{x})^{2} + s^3 \Theta ^3
\left(\frac{x-x_0}{a}\right)^2 v^{2}\right) e^{{2s\gamma}}  dxdt\le
C \int_0^T\int_{w}v^{2}\frac{1}{a} dxdt
\]
for all $s \ge s_0$.
Here $\Theta$ and $\gamma$ are as \eqref{c_1} and \eqref{gammabo}, respectively.
\end{Lemma}

The proof of the previous lemma is similar to the one of Lemma \ref{lemma3} with the suitable changes, but we repeat here for the reader's convenience. Also in this case, we underline that for the proof a crucial role is played by the Carleman estimates stated in \cite[Theorem 3.2]{fm1} for  non degenerate parabolic problems with  {\it non smooth coefficient}. Again, to apply such a result equation \eqref{5.3'} is essential.
Another important result to prove Lemma \ref{lemma3''} is
the following Caccioppoli inequality for the non divergence case:
\begin{Proposition}[Caccioppoli's inequality]\label{caccio1}
Assume that either the function $a$ is such that the associated operator $A_2$ is weakly degenerate and \eqref{5.3'} holds or the function $a$ is such that $A_2$ is strongly degenerate.
Moreover, let $I'$ and $I$ two open subintervals of $(0,1)$ such
that $I'\subset \subset I \subset  (0,1)$ and $x_0 \not
\in \overline{I}$. Let $\varphi(t,x)=\Theta(t)\Upsilon(x)$, where
$\Theta$ is defined in \eqref{c_1} and
$
\Upsilon \in C([0,1],(-\infty,0))\cap
C^1([0,1]\setminus\{x_0\},(-\infty,0))
$
satisfies
$
|\Upsilon_x|\leq\displaystyle \frac{c}{\sqrt{a}} \mbox{ in }[0,1]\setminus\{x_0\},
$
for some $c>0$. Then, there exist two positive constants
$C$ and $s_0$ such that every solution $v \in \cal W_2$ of the
adjoint problem \eqref{h=01} satisfies, for all $s\geq s_0$,
\[
   \int_{0}^T \int _{I'}   (v_x)^2e^{2s\varphi } dxdt
    \ \leq \ C \int_{0}^T \int _{I}   v^2  \frac{1}{a}dxdt.
\]
\end{Proposition}

We omit the proof of the previous result since it is similar to the one of \cite[Proposition 5.4]{fm1}.

\begin{Remark}\label{prototipo}
Of course, our prototype for $\Upsilon$ is the function $\mu$
defined in \eqref{c_1'}. Indeed, if $\mu$ is as
in \eqref{c_1'}, then, by $\cite[Lemma 2.1]{fm}$,
\[
|\mu'(x)|=d_1\frac{|x-x_0|e^{R(x-x_0)^2}}{a(x)}=d_1\sqrt{\frac{|x-x_0|^2e^{2R(x-x_0)^2}}{a(x)}}\frac{1}{\sqrt{a(x)}}\leq
c\frac{1}{\sqrt{a(x)}}.
\]
\end{Remark}

\begin{proof}[Proof of Lemma \ref{lemma3''}]Assume, first of all, that $\omega=(\alpha, \beta) \subset (0,1)$ is such that $x_0 \in \omega$.

Take $\bar\lambda_i, \tilde\lambda_i, \bar\beta_1$, ($i=1,2$), and the smooth functions $\xi, \eta$ as in the proof of Theorem \ref{Cor1}.
    Define $w:= \xi v$, where $v$ solves \eqref{h=01}, where, in this case,  $A_1$ is replaced by $A_2$.
   Hence, $w$   satisfies
  \[
\begin{cases}
w_t + a  w_{xx} = a (\xi _{xx} v + 2\xi _x  v_x )=:f,&
(t,x) \in(0, T)\times (0,1), \\
w(t,0)= w(t,1)=0, & t \in (0,T).
\end{cases}
\]
Applying Theorem \ref{Cor2}, we have
\begin{equation}\label{car_nondiv}
\begin{aligned}
&\int_0^T \int_0^1\Big( s \Theta   (w_x)^2 + s^3 \Theta^3
   \left(\frac{x-x_0}{a}\right)^2\ w^2 \Big)
    e^{2s \gamma} \, dx dt
    \\& \le  C\left(\int_0^T\int_0^1 \frac{f^{2}}{a}e^{2s\gamma(t,x)}dxdt+ \int_0^T\int_0^1 w^{2}dxdt\right)\\
&+ C\left( \int_0^T \int_0^{\beta_1}
 f^2 e^{2s \Phi_1(t,-x)}dxdt+ \int_0^T \int_{\lambda_2}^1 f^2e^{2s \Phi_2(t,x)}dxdt\right),
\end{aligned}
\end{equation}
for all $s \ge s_0$ and for a positive constant $C$. Then, using the definition of $\eta$  and in
particular the fact that  $\eta_x$ and  $\eta_{xx}$ are supported in
$\tilde \omega$, we can write
$
f^2= (( a \eta _x v )_x + \eta _x a v_x)^2 \le ( v^2+
(v_x)^2)\chi_{\tilde \omega},$
since the function $a'$  is bounded on $\tilde \omega$.
Hence,
applying \eqref{car_nondiv} and Proposition \ref{caccio1} with $I' = \tilde \omega$ and $I= (\alpha, \beta_1) \cup (\lambda_2, \beta)$, we get, as in \eqref{stimacar}
 \[
\begin{aligned}
&\int_0^T\int_{\tilde\lambda_1}^{\tilde\lambda_2}\left( s \Theta  (v_x)^{2} + s^3
\Theta ^3\left( \frac{x-x_0}{a}\right)^2 \frac{v^{2}}{a}\right) e^{{2s\gamma}} dxdt
\\
&\le C\int_0^T\int_\omega v^2e^{2s \gamma}dxdt+ C \int_0^T \int_{I} v^2 \frac{1}{a}dxdt \le C\int_0^T \int_\omega  v^2 \frac{1}{a}dxdt,
\end{aligned} \]
for a positive constant $C$.

The rest of the proof is similar to the last part of the one of Lemma \ref{lemma3}.
\end{proof}

Thanks to Lemma \ref{lemma3''} we have the next observability inequality in the case
of a regular final--time datum:
\begin{Lemma}\label{obser.regular1'}
Assume Hypotheses $\ref{hyp6.1}$. Then there
exists a positive constant $C_T$ such that every solution $v \in
\cal W_2$ of \eqref{h=01} satisfies
 \eqref{obser1.1}.
\end{Lemma}

The proof of the previous result follows as in \cite[Lemma 5.5]{fm1}, but we can refer also to the proof of Lemma \ref{obser.regular}.

\vspace{0.5cm}
Using Lemma \ref{obser.regular1'}, one can prove, as in \cite{cfr} or \cite{cfr1}, Proposition \ref{obser.1}.

\section{Final comments}
We conclude the paper with some comments about the estimates \eqref{carleman1} and \eqref{car'}.

 A Carleman estimate similar to \eqref{carleman1} for the problem in divergence form can follow by \cite[Theorem 4.1]{acf} at least in the
strongly degenerate case and if the initial datum is more regular.
Indeed, in this case, given $u_0\in H^1_a(0,1)$, $u$ is a
solution of \eqref{linear} if and only if the restrictions of $u$ to
$[0, x_0)$ and to $(x_0,1]$, $u_{|_{[0,x_0)}}$ and
$u_{|_{(x_0,1]}}$, are solutions to
\begin{equation}\label{me1}
\begin{cases}
u_t - A_1u   =h(t,x) \chi_{\omega}(x), & \ (t,x) \in (0,T) \times (0,x_0), \\
    u(t,0)=(au_x)(t,x_0)=0, & \  t \in (0,T),\\
     u(0,x)=u_0(x)_{|_{[0,x_0)}},
\end{cases}
\end{equation}
and
\begin{equation}\label{me2}
\begin{cases}
u_t -  A_1u   =h(t,x) \chi_{\omega}(x), & \ (t,x) \in (0,T) \times (x_0,1), \\
u(t,1)=(au_x)(t,x_0)=0, & \  t \in (0,T),\\
u(0,x)=u_0(x)_{|_{(x_0,1]}},
\end{cases}
\end{equation}
respectively. This fact is implied by the characterization of the
domain of $ A_1$ given in Propositions \ref{domain} and by
the Regularity Theorems \ref{prop} when the initial datum is
more regular. On the other hand if $u_0$ is only of class
$L^2(0,1)$, the solution is not sufficiently regular to verify the
additional condition at $(t,x_0)$ and this procedure cannot be
pursued.

Moreover, in the weakly degenerate case, the lack of
characterization of the domain of $A_1$ doesn't let us
consider a decomposition of the system in two disjoint systems like
\eqref{me1} and \eqref{me2}, in order to apply the results of
\cite{acf}, not even in the case of a regular initial datum.

Even if the problem is in non divergence form and the initial data is more regular, the above decomposition doesn't work. Indeed in this case, using the characterization of the domain of $A_2$, one has that $(au_x)(t,x_0)=0$ (this equality holds only in the strongly degenerate case, see Proposition \ref{domain4}). But, to our best knowledge, the only result on Carleman estimates in this field is for problems with {\it pure} Neumann boundary conditions, in the sense that $u_x(t,x_0)=0$, and with {\it more regular degenerate functions} (see \cite{f}), that we don't have in our hands.

\end{document}